\documentclass[12pt,a4paper, leqno]{article}
\usepackage[english]{babel}
\usepackage[utf8]{inputenc}
\usepackage[T1]{fontenc}
\usepackage{mathpazo}
\usepackage{geometry}
\usepackage{cite}
\usepackage{amsmath}
\usepackage{amssymb}
\usepackage{hyperref}
\usepackage{amsthm}
\usepackage{graphicx}

\def\IR{\mathbb{R}}
\def\IC{\mathbb{C}}

\def\IN{\mathbb{N}}
\def\IZ{\mathbb{Z}}
\def\A{\mathcal{A}}
\def\B{\mathcal{B}}

\def\D{\mathcal{D}}
\def\E{\mathcal{E}}
\def\H{\mathcal{H}}
\def\K{\mathcal{K}}

\def\a{\mathfrak{a}}
\def\b{\mathfrak{b}}

\def\1{\mathbbmss{1}}

\def\lra{\longrightarrow}
\def\downto{\downarrow}

\DeclareMathOperator{\supp}{supp}

\newcommand{\id}{\mathrm{id}}

\newcommand{\Deg}{\mathrm{Deg}}
\DeclareMathOperator{\sgn}{sgn}
\renewcommand{\Re}{\operatorname{Re}}
\renewcommand{\Im}{\operatorname{Im}}
\def\loc{\mathrm{loc}}

\providecommand{\abs}[1]{\lvert#1\rvert}
\providecommand{\norm}[1]{\lVert#1\rVert}

\renewcommand{\epsilon}{\varepsilon}
\renewcommand{\phi}{\varphi}

\makeatletter
\def\blfootnote{\gdef\@thefnmark{}\@footnotetext}
\makeatother

\newtheoremstyle{Beispiel}{}{}{}{}{\bfseries}{:}{ }{}
\theoremstyle{Beispiel}
\newtheorem{example}{Example}[section]
\newtheorem*{remark}{Remark}

\newtheoremstyle{Satz}{}{}{\itshape}{}{\bfseries}{:}{\newline}{}
\theoremstyle{Satz}
\newtheorem{proposition}[example]{Proposition}
\newtheorem{definition}[example]{Definition}
\newtheorem{theorem}[example]{Theorem}
\newtheorem{lemma}[example]{Lemma}
\newtheorem{corollary}[example]{Corollary}

\begin{document}
\frenchspacing 
\selectlanguage{english}
\title{Uniqueness of form extensions and domination of semigroups}
\author{Melchior Wirth}
\date{}

\maketitle

\begin{abstract}
In this article, we study questions of uniqueness of form extension for certain magnetic Schrödinger forms. The method is based on the theory of ordered Hilbert spaces and the concept of domination of semigroups. We review this concept in an abstract setting and give a characterization in terms of the associated forms. Then we use it to prove a theorem that transfers uniqueness of form extension of a dominating form to that of a dominated form. This result is applied in two concrete situations: magnetic Schrödinger forms on graphs and on domains in Euclidean space.
\end{abstract}


\begingroup
\def\addvspace#1{}
\tableofcontents
\endgroup

\section{Introduction}
Uniqueness of extension of formal differential expressions has long been of interest. In particular, in connection to quantum mechanics it is an object of study whether a given differential expression for the Schrödinger operator determines a unique self-adjoint realization and hence a unique time evolution via its unitary group.

Questions of this type have been studied at least since the '50s intensively, producing an immense output of results about essential self-adjointness and related questions as the coincidence of minimal and maximal form of a given Schrödinger operator. While the case of the Schrödinger operator for a free particle, given by $-\Delta$, is object of classical theory of elliptic differential operators (with constant coefficients) and quite easy to handle, Schrödinger operators with electric and magnetic potential are far more complicated and interesting objects of study.

The Schrödinger operator for a particle in an electric field with potential $V$ and a magnetic field with potential $b$ is formally given by
\begin{align*}
H_{b,V}=(\nabla-ib)^2+V.
\end{align*}
A closer study reveals that many results on uniqueness of realizations of this magnetic Schrö\-din\-ger operator crucially rely on the same technique, domination of semigroups, a concept going back to Kato (\cite{Kat72}), Simon (\cite{Sim77}) and Hess, Schrader, Uhlenbrock (\cite{HSU}). It is a generalization of Kato's inequality
\begin{align*}
\Delta\abs{u}\geq\Re(\overline{\sgn u}\Delta u),
\end{align*}
where the Laplacian on the right-hand side is replaced by the magnetic Schrö\-din\-ger operator $H_{b,V}$:
\begin{align*}
\Delta\abs{u}\geq \Re(\overline{\sgn u}H_{b,V}u).
\end{align*}
It was a key observation by Simon that this inequality is essentially a statement about semigroups, namely, under some conditions on the domains of definition, the generalized Kato's inequality for $\Delta$ and $H_{b,V}$ is equivalent to
\begin{align*}
\abs{e^{tH_{b,V}}u}\leq e^{t\Delta}\abs{u}.
\end{align*}
In this light, Kato's inequality for the Laplacian is a reformulation of the fact that the generated heat semigroup is positivity preserving, that is, $e^{t\Delta}u\geq 0$ for all $u\geq 0$.

Later, Ouhabaz (\cite{Ouh96}) gave a characterization of domination of semigroups in which the distributional Kato's inequality was replaced by an inequality for the associated quadratic forms $\a$, $\b$:
\begin{align*}
\Re \a(u,v\sgn u)\geq\b(\abs{u},v).
\end{align*}
Moreover, he observed that the concept of domination can also be applied to the case when $\a$ is defined on vector-valued $L^2$-functions.

The utility of domination is grounded in the fact that it allows to transfer uniqueness statements of the dominating operator to such of the dominated operator. It was the aim of this article to give a precise formulation and proof of this often applied concept. More concretely we prove in our main theorem (Theorem \ref{uniqueness_forms}) that the coincidence of minimal and maximal form for the dominating operator implies the coincidence of minimal and maximal form for the dominated operator.

There is a wide range of possible applications. As two examples, we study magnetic Schrödinger forms on vector bundles over graphs as defined in \cite{MT15} and magnetic Schrödinger forms on domains in Euclidean space. A particular strength of this approach in the Euclidean case is that it is quite robust with respect to the regularity of the magnetic and electric potential $b$ and $V$. Indeed, domination follows from a pointwise inequality, and except for the closedness of the Schrödinger forms, regularity will not be much of a concern (in particular, we will only make integrability assumptions on the coefficients and no differentiability or Lipschitz property is demanded).

However, we believe that our results are not limited to these two cases, but might as well be applicable in situations as diverse as:
\begin{itemize}
\item magnetic forms on Riemannian manifolds (cf. \cite{Sch01}, \cite{BMS02}),
\item magnetic forms on resistance spaces (cf. \cite{HR15}), in particular, magnetic forms on fractals (cf. \cite{HT13}),
\item forms associated to the Hodge-de Rham Laplacian on $p$-forms on a Riemannian manifold (cf. \cite{Gaf51}),
\item forms associated to higher-order de-Rham Laplacians on graphs.
\end{itemize}

The organization of this article is as follows: In Chapter \ref{chap_positivity} we review the basics of ordered Hilbert spaces and domination of operators, including an abstract version of the characterization of domination of semigroups in terms of the associated forms (Theorem \ref{thm_char_domination}). In Chapter \ref{chap_uniqueness_criterion} we prove the main theorem of this article, a criterion for form uniqueness in terms of domination. In Chapter \ref{Applications} we discuss the above mentioned applications, namely Schrödinger forms on vector bundles over graphs in Section \ref{Magnetic Schrödinger operators on graphs} and Schrödinger forms on domains in Euclidean space in Section \ref{Magnetic Schrödinger operators in Euclidean space}. In the Appendix \ref{appendix} some basic facts about forms and semigroups on Hilbert spaces are collected.

\textbf{Acknowledgment: }This article is based on the author's Master's thesis at Friedrich Schiller University Jena. He wants to express his gratitude to Professor Dr. Daniel Lenz and Marcel Schmidt for supervising the work on this thesis. Especially, he would like to thank Marcel for pointing him from questions of essential self-adjointness to form uniqueness in proposing a first version of the main result. Finally, the financial aid of the German National Academic Foundation (Studienstiftung des deutschen Volkes) is gratefully acknowledged.

\section{Positivity in Hilbert spaces}\label{chap_positivity}
\subsection{Positive cones}
In this section we collect some basics about order structures in Hilbert spaces induced by positive cone. All the material here is certainly well-known---see \cite{Nem03} for a (very short) introduction. To make this work self-contained, we included proofs of the elementary facts.

As it is convenient for order theory, we only deal with vector spaces over $\IR$ in this section.

\begin{definition}[Positive Cone]
Let $\K$ be a Hilbert space. A closed, non-empty subset $\K^+$ of $\K$ is called positive cone if
\begin{itemize}
\item[(P1)]$\K^+ +\K^+\subset \K^+$,
\item[(P2)]$ \alpha\K^+\subset \K^+$ for all $ \alpha\geq 0$,
\item[(P3)]$\langle\K^+,\K^+\rangle\geq 0$.
\end{itemize}
We will write $g_1\leq g_2$ if $g_2-g_1\in\K^+$.\\
The positive cone $\K^+$ is said to be self-dual if 
\begin{align*}
\K^+=\{g\in \K\mid \langle g,h\rangle\geq 0\text{ for all }h\in\K^+\}.
\end{align*}
The positive cone $\K^+$ is called an isotone projection cone if the projection $P_{\K^+}$ onto $\K^+$ is monotone increasing with respect to $\leq$, that is, $g\leq h$ implies $P_{\K^+}(g)\leq P_{\K^+}(h)$ for $g,h\in\K$.
\end{definition}

\begin{remark}
\begin{itemize}
\item If $\K^+\subset\K$ is a positive cone, the dual cone is defined as
\begin{align*}
(\K^+)^\circ=\{g\in\K\mid\langle g,h\rangle\geq 0\text{ for all }h\in \K^+\}.
\end{align*}
Hence a cone is self-dual if and only if it coincides with its dual. Sometimes the dual cone is also called polar cone (and self-dual cones are called self-polar), but the usual convention seems to be that the polar cone of $\K^+$ is $-(\K^+)^\circ$.
\item For self-dual cones it is redundant to assume that they are closed. Indeed,
\begin{align*}
\K^+=\bigcap_{h\in\K^+}\{g\in\H\mid \langle g,h\rangle\geq 0\}
\end{align*}
is closed as the intersection of inverse images of a closed set under continuous functions.\\
In general this is no longer true, as the following example shows: Let $\K=\ell^2$ and $\K^+=\{f\in\ell^2\mid f\geq 0,\,\supp f\text{ finite}\}$. Then $\K^+$ satisfies (P1) -- (P3), but is obviously not closed.
\item Some authors do not assume property (P3) in the definition of positive cones. However, in the case of self-dual cones we are mainly interested in it is automatically satisfied. Our terminology is in accordance with that of \cite{Ber} and \cite{HSU}, our main sources for Section \ref{sec_Domination_of_operators}.
\item The projection $P_C$ onto a closed, convex subset $C$ of a Hilbert space $H$ maps $x\in H$ to the unique element $P_C(x)\in C$ that satisfies $\norm{P_C(x)-x}=d(x,C)$. It is characterized as the unique $z\in C$ that satisfies $\langle x-z,y-z\rangle\leq 0$ for all $y\in C$.
\end{itemize}
\end{remark}

In most cases we will be interested in the following example.
\begin{example}
Let $(X,\B,m)$ be a measure space. Then
\begin{align*}
L^2(X,m)^+=\{f\in L^2(X,m)\mid f\geq 0\;m\text{-almost everywhere}\}
\end{align*}
is a self-dual isotone projection cone in $L^2(X,m)$.
\end{example}

\begin{remark}
It can be shown (cf. \cite{Pen76}, Corollary II.4) that all self-dual isotone projection cones arise in this way. Actually, Penne'y result is even a little stronger: If $\K^+\subset \K$ is a self-dual cone that induces a lattice ordering on $\K$, there is a compact space $X$, a regular finite Borel measure $\mu$ on $X$, and a unitary $U\colon \K\lra L^2(X,m)$ such that $U\K^+=L^2(X,m)^+$. That self-dual isotone projection cones indeed satisfy the assumption of this statement is content of Proposition \ref{isotone_latticial}.
\end{remark}

\begin{lemma}
Let $\K$ be a Hilbert space and $\K^+\subset\K$ a positive cone. Then the relation $\leq$ is a partial order on $\K$ satisfying
\begin{itemize}
\item[(a)]$f\leq g$ implies $f+h\leq g+h$,
\item[(b)]$f\leq g$ implies $ \alpha f\leq  \alpha g$
\end{itemize}
for all $f,g,h\in\K,\, \alpha\geq 0$.
\end{lemma}
\begin{proof}
Reflexivity, transitivity and the properties (a) and (b) of $\leq$ are immediate from (P1), (P2). For anti-symmetry let $g,h\in\K$ such that $g\leq h$ and $h\leq g$. Then (P3) implies
\begin{align*}
\norm{g-h}^2=-\langle \underbrace{g-h}_{\in\K^+},\underbrace{-(g-h)}_{\in\K^+}\rangle\leq 0,
\end{align*}
hence $g=h$.
\end{proof}

An important tool when dealing with positive cones is Moreau's Theorem (cf. \cite{M62}). It shows that in Hilbert spaces with a positive cone there is an abstract analog of the decomposition of a function into positive and negative part. In the original form it deals with two mutually polar cones, but we will need only the version for a self-dual cone as stated below.
\begin{theorem}\label{Moreau}
Let $\K$ be a Hilbert space and $\K^+\subset\K$ a self-dual positive cone. For $g,h_1,h_2\in\K$ the following statements are equivalent:
\begin{itemize}
\item[(i)]$g=h_1-h_2,\;h_1,h_2\in\K^+,\;\langle h_1,h_2\rangle=0$
\item[(ii)]$h_1=P_{\K^+}(g), h_2=P_{\K^+}(-g)$.
\end{itemize}
\end{theorem}
\begin{proof}
(i)$\implies$(ii): For all $u\in\K^+$ we have
\begin{align*}
\langle g-h_1, u-h_1\rangle=\langle -h_2,u-h_1\rangle=-\langle h_2, u\rangle\leq 0
\end{align*}
Thus, $h_1=P_{\K^+}(g)$ by the characterization of the projection onto a closed, convex set. Analogously, $h_2=P_{\K^+}(-g)$ follows.

(ii)$\implies$(i): Since $h_1=P_{\K^+}(g)$, we have $\langle g-h_1,u-h_1 \rangle\leq 0$ for all $u\in\K^+$. In particular, $\langle g-h_1,h_1\rangle\geq 0$ for $u=0$ and $\langle g-h_1, h_1\rangle\leq 0$ for $u=2 h_1$, hence $\langle g-h_1,h_1\rangle=0$.\\
Let $v=h_1-g$. For all $u\in\K^+$ we have
\begin{align*}
\langle u,v\rangle=\langle u,h_1-g\rangle=\langle v-h_1,h_1-g\rangle\geq 0.
\end{align*}
Thus, $v\in\K^+$. On the other hand we have for all $u\in\K^+$:
\begin{align*}
\langle v-(-g), u-v\rangle=\langle h_1,u-v\rangle=\langle h_1,u\rangle\geq 0.
\end{align*}
Hence, $v=P_{\K^+}(-g)=h_2$.
\end{proof}

\begin{definition}[Riesz space]
A vector space $E$ with a partial order $\leq$ is called ordered vector space if for all $f,g,h\in E$, $ \alpha\geq 0$, the following properties hold:
\begin{itemize}
\item[(R1)]$f\leq g$ implies $f+h\leq g+h$,
\item[(R2)]$f\leq g$ implies $ \alpha f\leq  \alpha g$.
\end{itemize}
If additionally
\begin{itemize}
\item[(R3)]$\{f,g\}$ has a least upper bound $f\vee g$
\end{itemize}
holds for all $f,g\in E$, then $E$ is called a Riesz space.\\
A subspace $F$ of a Riesz space $E$ is called sublattice if $f,g\in F$ implies $f\vee g\in F$.
\end{definition}

\begin{remark}
\begin{itemize}
\item Let $E$ be an ordered vector space. If a least upper bound (or a greatest lower bound) of $\{f,g\}\subset E$ exists, it is necessarily unique.
\item The greatest lower bound of $\{f,g\}$ will be denoted by $f\wedge g$ (if it exists).
\end{itemize}
\end{remark}

\begin{definition}[Positive and negative part]
Let $E$ be an ordered vector space. For $f\in E$ the positive and negative part are defined as $f^{\pm}=(\pm f)\vee 0$ if they exist. In this case the absolute value is defined as $\abs{f}=f^++f^-$.
\end{definition}

We collect some basic properties about suprema and infima in ordered vector spaces. They can be found for example in \cite{LZ}, Thms. 11.5, 11.7, 11.8.
\begin{lemma}\label{basic_Riesz}
Let $E$ be an ordered vector space. For $f,g,h\in E$ and $ \alpha\geq 0$ the following properties hold:
\begin{itemize}
\item[(a)] If $f\vee g$ exists, then $(f+h)\vee (g+h)$ exists and $(f+h)\vee (g+h)=(f\vee g)+h$.
\item[(b)] If $f\vee g$ exists, then $( \alpha f)\vee ( \alpha g)$ exists and $( \alpha f)\vee ( \alpha g)= \alpha (f\vee g)$.
\item[(c)] If $f\vee g$ exists, then $(-f)\wedge (-g)$ exists and  $(-f)\wedge (-g)=-(f\vee g)$.
\item[(d)] If $(f-g)^{+}$ and $(f-g)^-$ exist, then $f\vee g$ and $f\wedge g$ exist, and $f\vee g=\frac 1 2(f+g+\abs{f-g})$, $f\wedge g=\frac 1 2(f+g-\abs{f-g})$.
\item[(e)] If $f\vee g$, $f\wedge g$ exist, then $f\vee g+f\wedge g=f+g$ and $f\vee g-f\wedge g=\abs{f-g}$.
\end{itemize}
Additionally, one can replace the suprema in (a) and (b) by infima.
\end{lemma}
\begin{proof}
The proofs are all elementary and quite similar. We just show the first part of (d) as an example.\\
Let $f,g\in E$ such that $(f-g)^+,(f-g)^-=(g-f)^+$ exist. Then $(f-g)^+\geq f-g,\,(g-f)^+\geq 0$ imply
\begin{align*}
\frac 1 2(f+g+\abs{f-g})=\frac 1 2(f+g+(f-g)^++(g-f)^+)\geq \frac 1 2(f+g+f-g)=f
\end{align*}
and analogously $\frac 1 2(f+g+\abs{f-g})\geq g$.

Now let $h$ be an upper bound of $\{f,g\}$. Then $h-f\geq g-f$ and $h-f\geq 0$, hence $h-f\geq (g-f)^+$, that is, $h\geq f +(g-f)^+$. Changing the roles of $f$ and $g$ we obtain $h\geq g+(f-g)^+$. Combining both inequalities gives
\begin{align*}
h\geq \frac 1 2(f+(g-f)^+)+\frac 1 2(g+(f-g)^+)=\frac 1 2(f+g+\abs{f-g}).
\end{align*}
Therefore, $\frac 1 2(f+g+\abs{f-g})$ is a least upper bound of $\{f,g\}$.
\end{proof}

In particular, in Riesz spaces all infima $f\wedge g$ exist.

As a corollary of (d) we notice that it suffices to show the existence of suprema for sets of the form $\{f,0\}$ for all $f\in E$ instead of all subsets with two elements in order to prove that $E$ is a Riesz space.
\begin{corollary}\label{positive_part_implies_Riesz}
Let $E$ be an ordered vector space such that $f^+$ exists for all $f\in E$. Then $E$ is a Riesz space.
\end{corollary}

The connection between Riesz spaces and positive cones in Hilbert spaces is given by the following result by Isac and Németh (\cite{IN}, Proposition 3).
\begin{proposition}\label{isotone_latticial}
Let $\K$ be a Hilbert space and $\K^+\subset\K$ a self-dual isotone projection cone. Then $(\K,\leq)$ is a Riesz space.
\end{proposition}

The original proof is quite complicated, but as our situation is a little more restrictive than that in \cite{IN}, we can give a simple proof in the form of the following lemma combined with Corollary \ref{positive_part_implies_Riesz}.

\begin{lemma}\label{positive_part_projection}
Let $\K$ be a Hilbert space, $\K^+\subset \K$ a self-dual isotone projection cone, and $g\in\K$. Then $P_{\K^+}(g)$ is a least upper bound of $\{0,g\}$.
\end{lemma}
\begin{proof}
Let $g\in\K$. By Moreau's Theorem \ref{Moreau} we have
\begin{align*}
g=P_{\K^+}(g)-P_{\K^+}(-g)\leq P_{\K^+}(g).
\end{align*}
Hence, $P_{K^+}(g)$ is an upper bound for $\{0,g\}$. Now let $h$ be an upper bound for $\{0,g\}$. By isotonicity we have $P_{\K^+}(g)\leq P_{\K^+}(h)=h$. Thus, $P_{\K^+}(g)$ is the least upper bound of $\{0,g\}$.
\end{proof}

\begin{lemma}\label{norm_Riesz_monotone}
Let $\K$ be a Hilbert space and $\K^+\subset \K$ a self-dual isotone projection cone. Then 
\begin{align*}
\norm\cdot\colon \K^+\lra[0,\infty)
\end{align*}
is monotone increasing and $\norm{g}=\norm{\abs{g}}$ for all $g\in\K$.
\end{lemma}
\begin{proof}
Let $g,h\in \K^+$ such that $g\leq h$. Then we have
\begin{align*}
0\leq \langle h-g,g\rangle=\langle h,g\rangle-\norm g^2\leq \norm h\norm g-\norm g^2,
\end{align*}
hence $\norm g\leq\norm h$.\\
Moreover,
\begin{align*}
\norm{g}^2&=\langle g^+-g^-,g^+-g^-\rangle=\langle g^++g^-,g^++g^-\rangle=\norm{\abs{g}}^2
\end{align*}
since $\langle g^+,g^-\rangle=0$ by Moreau's Theorem \ref{Moreau}.
\end{proof}

Having studied some of the basic properties, we will briefly turn to forms on ordered Hilbert spaces that are compatible with the order structure.
\begin{definition}[Positive form]
Let $\K$ be a Hilbert space and $\K^+\subset\K$ a positive cone. A form $(\b,D(\b))$ in $\K$ is called positive if $P_{\K^+}D(\b)\subset D(\b)$ and
\begin{align*}
\b(P_{\K^+}(g),P_{\K^+}(-g))\leq 0
\end{align*}
for all $g\in D(\b)$.
\end{definition}

By Ouhabaz' invariance criterion (Proposition \ref{invariance_ouhabaz}), a form is positive if and only if the associated semigroup preserves the positive cone $\K^+$, explaining the term ``positive form''. This notion should not be confused with that of forms with lower bound $0$, which are also sometimes called ``positive''.

\begin{lemma}\label{positive_form_Riesz}
Let $\K$ be a Hilbert space, $\K^+\subset \K$ a self-dual isotone projection cone, and $(\b,D(\b))$ a form in $\K$. Then $\b$ is positive if and only if $\abs{D(\b)}\subset D(\b)$ and $\b(\abs{g})\leq \b(g)$ for all $g\in D(\b)$.
\end{lemma}
\begin{proof}
Let $\b$ be a positive form and $g\in D(\b)$. Then $\abs{g}=g^++(-g)^+\in D(\b)$ and
\begin{align*}
\b(\abs{g})&=\b(g^++(-g)^+)\\
&=\b(g^+)+\b((-g)^+)+2\b(g^+,(-g)^+)\\
&\leq \b(g^+)+\b((-g)^+)-2\b(g^+,(-g)^+)\\
&=\b(g^+-(-g)^+)\\
&=\b(g)
\end{align*}
since $\b(g^+,(-g)^+)\leq 0$.\\
Conversely assume that $\abs{D(\b)}\subset D(\b)$ and $\b(\abs{g})\leq \b(g)$ for all $g\in D(\b)$. Let $g\in D(\b)$. Then $g^+=\frac 1 2(g+\abs{g})\in D(\b)$ and
\begin{align*}
\b(g^+,(-g)^+)&=\frac 1 4(\b(g^++(-g)^+)-\b(g^+-(-g)^+))\\
&=\frac 1 4(\b(\abs{g}-\b(g))\\
&\leq 0.\qedhere
\end{align*}
\end{proof}

\begin{lemma}\label{pos_form_lattice}
Let $\K$ be a Hilbert space, $\K^+\subset\K$ a self-dual isotone projection cone, and $(\b,D(\b))$ a positive form in $\K$ with lower bound $- \lambda\in\IR$. 
\begin{itemize}
\item[(a)]The form domain $D(\b)$ is a sublattice of $\K$.
\item[(b)]For any $\alpha\geq \lambda$, the form $\b_ \alpha:=\b+ \alpha\langle\cdot,\cdot\rangle$ satisfies
\begin{align*}
\b_ \alpha(g\wedge h),\b_ \alpha(g\vee h)\leq \b_ \alpha(g)+\b_ \alpha(h)
\end{align*}
for all $g,h\in D(\b)$.
\end{itemize}
\end{lemma}
\begin{proof}
Let $g,h\in D(\b)$. By Lemma \ref{basic_Riesz} and \ref{positive_form_Riesz} we have
\begin{align*}
g\wedge h=\frac 1 2(g+h-\abs{g-h})\in D(\b)
\end{align*}
and analogously for $g\vee h$. Hence, $D(\b)$ is a sublattice of $\K$.\\
Since $\langle P_{\K^+}(g),P_{\K^+}(-g)\rangle=0$ by Moreau's Theorem \ref{Moreau}, the form $b_ \alpha$ is positive. Moreover, $\b_ \alpha(u+v)\geq 0$ implies $-2\b_ \alpha(u,v)\leq \b_ \alpha(u)+\b_ \alpha(v)$ for all $u,v\in D(\b)$. With the aid of this inequality, the positivity of $\b_ \alpha$ and the parallelogram identity we obtain
\begin{align*}
\b_ \alpha(g\wedge h)&=\frac 1 4\b_ \alpha(g+h-\abs{g-h})\\
&=\frac 1 4(\b_\alpha(g+h)+\b_\alpha\abs{g-h})-2\b_\alpha(g+h,\abs{g-h}))\\
&\leq \frac 1 2(\b_\alpha(g+h)+\b_\alpha(\abs{g-h})\\
&\leq\frac 1 2(\b_\alpha(g+h)+\b_\alpha(g-h))\\
&=\b_\alpha(g)+\b_\alpha(h).
\end{align*}
The result for $\b_\alpha(g\vee h)$ follows similarly.
\end{proof}

\begin{remark}
The theory developed so far is valid for real Hilbert spaces since that completely serves our purposes. To incorporate complex Hilbert spaces, one can proceed as follows:\\
Every complex Hilbert space $\K$ becomes a real Hilbert space $\K_{\mathrm{r}}$ when equipped with the inner product $\langle\cdot,\cdot\rangle_{\mathrm{r}}=\Re\langle\cdot,\cdot\rangle$ and a positive cone $\K^+$ in $\K$ is also a positive cone in $\K_{\mathrm{r}}$. However, self-duality of $\K^+$ is not preserved, but $\K_{\mathrm{r}}$ decomposes as
\begin{align*}
\K_{\mathrm{r}}=\K^J\oplus i\K^J
\end{align*}
with $\K^J=\K^+-\K^+$, and $\K^+$ is a self-dual cone in $\K^J$. Therefore, every element $g\in \K^+$ has a unique decomposition as
\begin{align*}
g=g_1-g_2+i(g_3-g_4)
\end{align*}
with $g_1,\dots,g_4\in\K^+$ and $\langle g_i,g_j\rangle=0$ for $i\neq j$.\\
This decomposition yields an anti-unitary involution $J$ via
\begin{align*}
J\colon \K^J\oplus i\K^J\lra\K^j\oplus  i\K^J,\,g+ih\mapsto g-ih.
\end{align*}
Positive forms on $\K$ are real in the sense that $JD(\b)=D(\b)$ and $\b(Jg)=\b(g)$.\\
Thus, there is no loss of generality when dealing exclusively with real Hilbert spaces as we can always restrict to $\K^J$ in the complex case.
\end{remark}

\subsection{Domination of operators}\label{sec_Domination_of_operators}
Domination of operators is a way to compare two operators. It is a generalization of the first Beurling-Deny criterion, which compares an operator to itself. In connection with questions of essential self-adjointness, it probably occurred first in the form of Kato's inequality comparing Schrödinger operators with and without magnetic field (cf. \cite{Kat72}).

Abstractly, it was first defined for operators on $L^2$-spaces (cf. \cite{Sim77, Sim79b}) and later generalized by Ouhabaz to the case of one operator acting on a vector valued $L^2$-space (cf. \cite{Ouh96}). In the form presented here, which follows \cite{HSU} and \cite{Ber}, it can be applied to two operators on arbitrary Hilbert spaces as long as there is a so called symmetrization map between them.

Domination will be the main tool in this article to transfer form uniqueness of one operator to form uniqueness of another, dominated operator. Often the question of uniqueness of form extensions is easier or already proven for the dominating form. Applications include magnetic Schrödinger operators on domains and graphs (cf. Chapter \ref{Applications}).

The main objective of this section after introducing the notion of domination of operators (Definition \ref{def_domination}) will be to give a characterization of domination of semigroups in terms of the associated forms (Theorem \ref{thm_char_domination}).

Throughout this section $\K$ shall denote a real Hilbert space and $\H$ a Hilbert space, either real or complex.

\begin{definition}[Symmetrization]
Let $\K^+\subset \K$ be a positive cone. A map $S\colon \H\lra\K^+$ is called absolute mapping if
\begin{itemize}
\item[(S1)]$\abs{\langle f_1,f_2\rangle}\leq\langle S(f_1),S(f_2)\rangle$ for all $f_1,f_2\in\H$ and equality if $f_1=f_2$.
\end{itemize}
An absolute mapping is called absolute pairing or symmetrization if
\begin{itemize}
\item[(S2)]For any $g\in\K^+$ and $f_1\in\H$ there is an $f_2\in\H$ such that $g=S(f_2)$ and 
\begin{align*}
\langle f_1,f_2\rangle=\langle S(f_1),S(f_2)\rangle=\langle S(f_1),g\rangle.
\end{align*}
In this case $f_1$ and $f_2$ are called $g$-paired or simply paired.
\end{itemize}
\end{definition}

Th following lemma shows that we have already encountered a natural example of an absolute mapping in the last section.
\begin{lemma}\label{absolute_mapping_Riesz}
Let $\K^+\subset\K$ be a self-dual, isotone projection cone. Then $\abs\cdot\colon\K\lra\K^+$ is an absolute mapping.
\end{lemma}
\begin{proof}
Let $g,h\in\K$. Then we have
\begin{align*}
\abs{\langle g,h\rangle}&=\abs{\langle g^+-g^-,h^+-h^-\rangle}\\
&=\abs{\langle g^+,h^+\rangle-\langle g^-,h^+\rangle-\langle g^+,h^-\rangle+\langle g^+,h^+\rangle}\\
&\leq \langle g^+,h^+\rangle+\langle g^-,h^+\rangle+\langle g^+,h^-\rangle+\langle g^+,h^+\rangle\\
&=\langle g^++g^-,h^++h^-\rangle\\
&=\langle \abs{g},\abs{h}\rangle.
\end{align*}
Equality in the case $g=h$ was already shown in Lemma \ref{norm_Riesz_monotone}.
\end{proof}

\begin{example}
The norm $\norm{\cdot}\colon\H\lra [0,\infty)$ is a symmetrization. For $\lambda>0$ and $f_1\in\H$ an element $f_2\in\H$ such that $f_1$ and $f_2$ are $\lambda$-paired is given by $f_2=\lambda\frac{f_1}{\norm{f_1}}$ if $f_1\neq 0$, and by $f_2=\lambda\xi$ for any $\xi\in\H$ with $\norm{\xi}=1$ if $f_1=0$.
\end{example}

\begin{example}\label{ex_sym_L2}
Let $(X,\B,m)$ be a measure space and $H$ a Hilbert space. Denote the norm on $H$ by $\abs\cdot$. Then $S\colon L^2(X,m;H)\lra L^2(X,m),\,(Sf)(x)=\abs{f(x)}$ is a symmetrization. Property (S1) is obvious. For (S2) choose $f_2=g\sgn_\xi f_1$, where $\sgn_\xi f_1$ is defined by
\begin{align*}
\sgn_\xi f_1(x)=\begin{cases}\frac{f_1(x)}{\abs{f_1(x)}}&\colon f_1(x)\neq 0\\\xi&\colon f_1(x)=0\end{cases}
\end{align*}
for some $\xi\in H$ with $\abs\xi=1$. Then $f_1$ and $f_2$ are $g$-paired.
\end{example}

\begin{example}
Let $X$ be a topological space, $m$ a Borel measure on $X$ and $E$ a Hermitian vector bundle over $X$, that is, a vector bundle with an inner product on the fibers that varies continuously with the base point (see \cite{MS74}, where the term \textit{Euclidean vector bundle} is used). Denote by $L^2(X,m;E)$ the space of $L^2$-sections in $E$. Then
\begin{align*}
S\colon L^2(X,m;E)\lra L^2(X,m)^+,\,(Sf)(x)=\abs{f(x)}_x
\end{align*}
is a symmetrization by the same arguments as in the example above.
\end{example}

As these examples suggest, a symmetrization can be understood as a generalization of the (pointwise) modulus of a function. In the following lemmas we collect some basic properties of the modulus that carry over to abstract symmetrizations.

\begin{lemma}\label{symmetrization_triangle}
Let $\K^+\subset\K$ be a positive cone, and $S\colon\H\lra\K^+$ a symmetrization.
\begin{itemize}
\item[(a)]The triangle inequality holds:
\begin{align*}
\langle S(f_1+f_2),g\rangle\leq \langle S(f_1)+S(f_2),g\rangle
\end{align*}
for all $f_1,f_2\in\H$ and $g\in\K^+$.
\item[(b)]The symmetrization $S$ is positive homogeneous:
\begin{align*}
S(\alpha f)=\abs{\alpha} S(f)
\end{align*}
for all $f\in\H$ and $\alpha\in\IC$.
\item[(c)]The symmetrization $S$ is positive definite: For all $f\in\H$, $S(f)=0$ if and only if $f=0$.
\end{itemize}
\end{lemma}

\begin{proof}
\begin{itemize}
\item[(a)]Let $f_1,f_2\in\H$ and $g\in\K^+$. Choose $f\in \H$ such that $S(f)=g$, $\langle S(f_1+f_2),S(f)\rangle=\langle f_1+f_2,f\rangle$. Then we have
\begin{align*}
\langle S(f_1+f_2),g\rangle&=\langle S(f_1+f_2),S(f)\rangle\\
&=\langle f_1+f_2,f\rangle\\
&\leq \langle S(f_1),S(f)\rangle+\langle S(f_2),S(f)\rangle\\
&=\langle S(f_1)+S(f_2),g\rangle.
\end{align*}
\item[(b)]Let $\alpha\in \IC,f\in\H$. Then we have
\begin{align*}
\norm{S(\alpha f)}^2=\norm{\alpha f}^2=\abs{\alpha}^2\norm{f}^2=\abs{\alpha}^2\norm{S(f)}^2.
\end{align*}
Next let $\lambda\geq 0$. Then we have
\begin{align*}
 \norm{S(\lambda f)-\lambda S(f)}^2&=\norm{S(\lambda f)}^2+\lambda^2\norm{f}^2-2\lambda\langle S(\lambda f),S(f)\rangle\\
&\leq 2\lambda^2\norm{f}^2-2\lambda \langle \lambda f,f\rangle\\
&=0.
\end{align*}
Moreover,
\begin{align*}
\norm{S(\alpha f)-S(\abs{\alpha}f)}^2&=2\abs{\alpha}^2\norm{f}^2-2\langle S(\alpha f),S(\abs{\alpha}f)\rangle\\
&\leq 2\abs{\alpha}^2\norm{f}^2-2\abs{\langle \alpha f,\abs{\alpha}f\rangle}\\
&=0.
\end{align*}
So we finally obtain $S(\alpha f)=S(\abs{\alpha}f)=\abs{\alpha}S(f)$.
\item[(c)]This is immediate from $\norm{S(f)}=\norm{f}$.\qedhere
\end{itemize}
\end{proof}

\begin{lemma}\label{symmetrization_Lipschitz}
Let $\K^+\subset \K$ be a positive cone and $S\colon\H\lra \K^+$ a symmetrization. Then $S$ is Lipschitz continuous.
\end{lemma}
\begin{proof}
For $f_1,f_2\in\H$ we have
\begin{align*}
\norm{S(f_1)-S(f_2)}^2&=\norm{S(f_1)}^2+\norm{S(f_2)}^2-2\langle S(f_1),S(f_2)\rangle\\
&=\norm {f_1}^2+\norm {f_2}^2-2\langle S(f_1),S(f_2)\rangle\\
&\leq \norm {f_1}^2+\norm {f_2}^2-2 \abs{\langle f_1,f_2\rangle}\\
&\leq\norm {f_1-f_2}^2.\qedhere
\end{align*}
\end{proof}

\begin{lemma}\label{abs_difference}
Let $\K^+\subset\K$ be a positive cone and $S\colon \H\lra\K^+$ a symmetrization. If $f_1,f_2\in\H$ such that $S(f_2)\leq S(f_1)$ and $f_1,f_2$ are $S(f_2)$-paired, then $S(f_1-f_2)=S(f_1)-S(f_2)$ and $f_1-f_2, f_2$ are $S(f_2)$-paired.
\end{lemma}
\begin{proof}
By the triangle inequality, $S(f_1-f_2)\geq S(f_1)-S(f_2)$ holds. Moreover we have
\begin{align*}
\norm{S(f_1-f_2)}^2&=\norm{f_1-f_2}^2\\
&=\norm{f_1}^2+\norm{f_2}^2-2\langle f_1,f_2\rangle\\
&=\norm{S(f_1)}^2+\norm{S(f_2)}^2-2\langle S(f_1),S(f_2)\rangle\\
&=\norm{S(f_1)-S(f_2)}^2.
\end{align*}
Let $g,g'\in\K^+$ such that $g\leq g'$ and $\norm g=\norm{g'}$. Then we have
\begin{align*}
\norm{g-g'}^2&=\norm{g}^2+\norm{g'}^2-2\langle g,g'\rangle\\
&=2\norm{g}^2-2\langle g,g+g'-g\rangle\\
&=-2\langle g,g'-g\rangle\\
&\leq 0.
\end{align*}
Hence $g=g'$. Applying this result to $g=S(f_1)-S(f_2), g'=S(f_1-f_2)$, we obtain the desired equality for $S(f_2-f_1)$.\\
Moreover,
\begin{align*}
\langle f_2-f_1,f_2\rangle=\langle S(f_2)-S(f_1),S(f_2)\rangle=\langle S(f_1-f_2),S(f_2)\rangle,
\end{align*}
hence $f_1-f_2$ and $f_2$ are paired.
\end{proof}

The next lemma will serve as a characterization of the central concept of this section, namely domination of operators.
\begin{lemma}\label{abs_char_domination}
Let $\K^+\subset \K$ be a positive cone and $S\colon\H\lra\K^+$ a symmetrization. For bounded operators $P$ (resp. $Q$) on $\H$ (resp. $\K$), the following are equivalent:
\begin{itemize}
\item[(i)]$\langle S(P f_1),g\rangle\leq \langle Q S (f_1),g\rangle$ for all $f_1\in \H,\,g\in\K^+$
\item[(ii)]$\Re\langle  Pf_1,f_2\rangle\leq \langle Q S(f_1),S(f_2)\rangle$ for all $f_1,f_2\in\H$
\item[(iii)]$\abs{\langle P f_1,f_2\rangle}\leq \langle Q S(f_1),S(f_2)\rangle$ for all $f_1,f_2\in\H$
\end{itemize}
Furthermore, if $\K^+$ is self-dual, these assertions are equivalent to
\begin{itemize}
\item[(iv)]$S(P f_1)\leq Q S(f_1)$ for all $f_1\in \H$.
\end{itemize}
\end{lemma}
\begin{proof}
(iii)$\implies$(ii): This is trivial.\\
(ii)$\implies$(i): Let $f_1\in \H, g\in\K^+$ and choose $f_2\in \H$ such that $Pf_1$ and $f_2$ are $g$-paired. Then we have
\begin{align*}
\langle S(Pf_1),g\rangle=\Re\langle S(Pf_1),g\rangle=\Re\langle Pf_1,f_2\rangle \leq \langle Q S(f_1),g\rangle.
\end{align*}
(i)$\implies$(iii): Let $g=S(f_2)$. Then we have
\begin{align*}
\abs{\langle Pf_1,f_2\rangle}\leq \langle S(P f_1),S(f_2)\rangle=\langle S(P f_1),g\rangle\leq \langle Q S(f_1),g\rangle=\langle Q S(f_1),S(f_2)\rangle.
\end{align*}
Next assume that $\K^+$ is self-dual and let $f_1\in \H$. Then $Q S(f_1)-S(Pf_1)\in \K^+$ if and only if
\begin{align*}
\langle Q S(f_1)-S(Pf_1),g\rangle\geq 0
\end{align*}
for all $g\in \K^+$. Thus, (iv) is equivalent to (i).
\end{proof}

\begin{definition}[Domination of operators]\label{def_domination}
If $P$ and $Q$ satisfy one of the equivalent assertions of Lemma \ref{abs_char_domination}, then $P$ is said to be dominated by $Q$. Likewise we say that a family of bounded operators $(P_\alpha)_{\alpha\in\A}$ is dominated by the family of bounded operators $(Q_\alpha)_{\alpha\in \A}$ if $P_\alpha$ is dominated by $Q_\alpha$ for all $\alpha\in \A$.
\end{definition}

\begin{example}\label{ex_positive_semigroup_dominating}
Let $(X,\B,m)$ be a measure space and $P\colon L^2(X,m)\lra L^2(X,m)$ a linear operator that leaves $L^2(X,m)^+$ invariant. Then $P$ is dominated by itself:\\
For all $f\in L^2(X,m)$, we have
\begin{align*}
\abs{Pf}=\abs{P(f^+-f^-)}\leq \abs{Pf^+}+\abs{Pf^-}=P f^++P f^-=P\abs{f}.
\end{align*}
Indeed, also the converse is true: If $P$ is dominated by itself, then $Pf=P\abs{f}\geq \abs{Pf}\geq 0$ for all $f\in L^2(X,m)^+$, hence $P L^2(X,m)^+\subset L^2(X,m)^+$.
\end{example}

We will give more interesting examples (in particular such that have different operators $P$ and $Q$) once we have a characterization of domination of semigroups in terms of the associated forms at hand.

But before we turn to this characterization, we present some basic algebraic properties of domination.

\begin{lemma}\label{domination_sums}
Let $\K^+\subset \K$ be a positive cone, $S\colon\H\lra\K^+$ a symmetrization, and $P_1,P_2$ (resp. $Q_1,Q_2$) bounded self-adjoint operators on $\H$ (resp. $\K$).
\begin{itemize}
\item[(a)]Let $\alpha_1,\alpha_2\in\IC$. If $P_i$ is dominated by $Q_i$, $i\in\{1,2\}$, then $\alpha_1P_1+\alpha_2P_2$ is dominated by $\abs{\alpha_1}Q_1+\abs{\alpha_2}Q_2$.
\item[(b)]If $\K^+$ is self-dual and $P_1$ is dominated by $Q_1$, then $Q_1$ preserves the cone $\K^+$.\\
If $P_i$ is dominated by $Q_i$, $i\in\{1,2\}$, and $Q_1$ preserves $\K^+$, then $P_1P_2$ is dominated by $Q_1 Q_2$.
\end{itemize}
\end{lemma}
\begin{proof}
\begin{itemize}
\item[(a)]We show that (iii) of Lemma \ref{abs_char_domination} is satisfied. Indeed, for all $f_1,f_2\in\H$ we have:
\begin{align*}
\abs{\langle (\alpha_1 P_1+\alpha_2 P_2)f_1,f_2\rangle}&\leq\abs{\alpha_1}\abs{\langle P_1 f_1,f_2\rangle}+\abs{\alpha_2}\abs{\langle P_2 f_1,f_2\rangle}\\
&\leq \abs{\alpha_1}\langle Q_1 S(f_1),S(f_2)\rangle+\abs{\alpha_2}\langle Q_2 S(f_1),S(f_2)\rangle\\
&=\langle (\abs{\alpha_1}Q_1+\abs{\alpha_2} Q_2) S(f_1),S(f_2)\rangle
\end{align*}
\item[(b)]Let $g\in \K^+$. If $\K^+$ is self-dual and $P_1$ is dominated by $Q_1$, then Lemma \ref{abs_char_domination} (i) implies
\begin{align*}
0\leq\langle S(P_1 f_1),h\rangle\leq \langle Q_1 g,h\rangle
\end{align*}
for all $h\in\K^+$ and $f_1\in \H$ such that $S(f_1)=g$. Hence, $Q_1 g\in\K^+$.\\
Next assume that $Q_1$ preserves $\K^+$ and that $P_i$ is dominated by $Q_i$, $i\in\{1,2\}$. Let $f\in \H$, $g\in\K^+$. Using \ref{abs_char_domination} (i), we obtain
\begin{align*}
\langle S(P_1P_2 f),g\rangle&\leq\langle Q_1 S(P_2f),g\rangle\\
&= \langle S(P_2f),Q_1 g\rangle\\
&\leq \langle Q_2 S(f),Q_1 g\rangle\\
&=\langle Q_1Q_2 S(f),g\rangle.
\end{align*}
Thus, $P_1 P_2$ is dominated by $Q_1 Q_2$.\qedhere
\end{itemize}
\end{proof}

The next proposition is the main tool in the characterization of domination of semigroups in terms of the associated forms. It is an abstract version of Lemma 3.2 in \cite{MVV}. The proof given there relies on pointwise consideration that are not applicable in our setting. Our proof is a bit lengthy, but elementary as it only uses the abstract properties of absolute mappings and isotone projection cones.

\begin{proposition}\label{domination_invariance}
Let $\K^+\subset\K$ be a self-dual positive cone, $S\colon \H\lra \K^+$ an absolute mapping, and $(P_t)$ (resp. $(Q_t))$ a semigroup on $\H$ (resp. $\K$).\\
Define the semigroup $(W_t)$ on $\H\oplus\K$ via
\begin{align*}
W_t(f,g)=(P_t f,Q_t g)
\end{align*}
for $t\geq 0,\,(f,g)\in\H\oplus\K$. Define the set
\begin{align*}
C=\{(u,v)\in\H\oplus\K\mid S(u)\leq v\}.
\end{align*}
\begin{itemize}
\item[(a)]The set $C$ is a closed, convex subset of $\H\oplus \K$.
\item[(b)]The semigroup $(P_t)$ is dominated by $(Q_t)$ if and only if $C$ is invariant under $(W_t)$.
\item[(c)]Let $g\in \K^+$ and $f_1\in \H$ with $g\leq S(f_1)$. Whenever there is an $f_2\in\H$ such that $f_1,f_2$ are $g$-paired, the projection $P_C$ onto $C$ satisfies
\begin{align*}
P_C(f_1,g)=\frac 1 2(f_1+f_2,S(f_1)+g).
\end{align*}
\item[(d)]If $\K^+$ is a self-dual isotone projection cone, the projection $P_C$ onto $C$ satisfies
\begin{align*}
P_C(f_1,g)=\frac 1 2(f_2, (S(f_1)\vee g+g)^+),
\end{align*}
for $f_1\in\H,g\in \K$ whenever there is an $f_2\in\H$ such that $f_1,f_2$ are $(S(f_1)\wedge g+S(f_1))^+$-paired.

\end{itemize}
\end{proposition}
\begin{proof}
\begin{itemize}
\item[(a)]Since $S$ is positive homogeneous and satisfies the triangle inequality (Lem\-ma \ref{symmetrization_triangle}), it is clear that $C$ is convex. By Lemma \ref{symmetrization_Lipschitz}, $S$ is continuous. Thus, $C$ is closed as the preimages of $\K^+$ under the continuous map
\begin{align*}
\phi\colon \H\oplus\K\lra \K,(u,v)\mapsto v-S(u).
\end{align*}
\item[(b)]First assume that $C$ is invariant under $(W_t)$.\\
Let $f\in \H$ and $t\geq 0$. Then we have $(f,S(f))\in C$, hence 
\begin{align*}
(P_t f, Q_t S(f))=W_t(f,S(f))\in C,
\end{align*}
that is, $S(P_t f)\leq Q_t S(f)$. Thus, $(P_t)$ is dominated by $(Q_t)$.

Conversely assume $(P_t)$ is dominated by $(Q_t)$.\\
Let $(u,v)\in C$. Since $\K^+$ is self-dual, the semigroup $(Q_t)$ leaves $\K^+$ invariant by Lemma \ref{domination_sums}. Thus, $S(u)\leq v$ implies $S(P_t u)\leq Q_t S(u)\leq Q_t v$. Hence, $W_t(u,v)=(P_tu,Q_tv)\in C$.
\item[(c)]Let $f_1\in \H$, $g\in\K^+$ such that $g\leq S(f_1)$ and assume there is an $f_2\in\H$ such that $f_1,f_2$ are $g$-paired. Define 
\begin{align*}
P(f_1,g)=\frac 1 2(f_1+f_2,S(f_1)+g).
\end{align*}
 The projection $(\hat{f_1},\hat g)$ of $(f_1,g)$ onto $C$ is characterized as the unique element in $C$ satisfying
\begin{align*}
\Re\langle (f_1,g)-(\hat{f_1},\hat g),(u,v)-(\hat{f_1},\hat g)\rangle\leq 0
\end{align*}
for all $(u,v)\in C$. We will show that $P(f_1,g)=(\hat f_1,\hat g)$.\\
Since $S((f_1+f_2))\leq S(f_1)+S(f_2)=S(f_1)+g$, $P(f_1,g)$ is an element of $C$. For all $(u,v)\in C$ we have
\begin{align*}
&\Re\langle (f_1,g)-P(f_1,g),(u,v)-P(f_1,g)\rangle\\
={}&\frac 1 4\Re\langle (f_1-f_2,g-S(f_1)),(2u-f_1-f_2,2v-S(f_1)-g)\rangle\\
={}&\frac 1 4\Re(\langle f_1-f_2,2u\rangle-\norm{f_1}^2+\norm{f_2}^2+2\langle g-S(f_1), v\rangle-\norm{g}^2+\norm{S(f_1)}^2)\\
={}&\frac 1 2\Re(\langle f_1-f_2,u\rangle+\langle g-S(f_1),v\rangle)
\end{align*}
Since $S(f_2)=g\leq S(f_1)$, we have $S(f_1-f_2)=S(f_1)-S(f_2)$ by Lemma \ref{abs_difference} and therefore 
\begin{align*}
\abs{\langle f_1-f_2,u\rangle}\leq \langle S(f_1-f_2),S(u)\rangle\leq \langle S(f_1)-S(f_2),v\rangle.
\end{align*}
This implies
\begin{align*}
\frac 1 2\Re(\langle f_1-f_2,u\rangle+\langle g-S(f_1),v\rangle)&\leq \frac 1 2\Re(\langle g-S(f_1),v\rangle+\abs{\langle f_1-f_2,u\rangle})\\
&\leq\frac 1 2\Re(\langle g-S(f_1),v\rangle+\langle S(f_1)-g,v\rangle)\\
&=0.
\end{align*}
Hence, $P(f_1,g)$ is the projection of $(f_1,g)$ on $C$.
\item[(d)]Let $f_1\in\H,g\in\K$ and $f_2\in \H$ such that $f_1,f_2$ are $(S(f_1)\wedge g+S(f_1))^+$-paired. Define
\begin{align*}
P(f_1,g)=\frac 1 2(f_2,(S(f_1)\vee g+g)^+).
\end{align*}
As in (c), we will show $P=P_C$ via the characterization of $P_C$ given above.\\
Since $\K^+$ is an isotone projection cone, $S(f_1)\wedge g+S(f_1)\leq g+S(f_1)\vee g$ implies $(S(f_1)\wedge g+S(f_1))^+\leq (S(f_1)\vee g+g)^+$, hence $P(f_1,g)\in C$.\\
So we have to show that
\begin{align*}
\Re\langle (f_1,g)-P(f_1,g),(u,v)-P(f_1,g)\rangle\leq 0
\end{align*}
for all $(u,v)\in C$.\\
We will evaluate the terms
\begin{align*}
I&=\Re\langle (f_1,g)-P(f_1,g),(u,v)\rangle
\intertext{and}
J&=\langle (f_1,g)-P(f_1,g),-P(f_1,g)\rangle
\end{align*}
separately.\\
Using $\abs{\langle f,\tilde f\rangle}\leq \langle S(f),S(\tilde f)\rangle$ for $f,\tilde f\in\H$, and $S(u)\leq v$, we obtain
\begin{align*}
I&=\Re\langle f_1-\frac 1 2 f_2,u\rangle+\langle g-\frac 1 2(S(f_1)\vee g+g)^+,v\rangle\\
&\leq \langle S(f_1-\frac 1 2f_2),S(u)\rangle+\langle g-\frac 1 2(S(f_1)\vee g+g)^+,v\rangle\\
&\leq \langle S(f_1-\frac 1 2f_2)+g-\frac 1 2(S(f_1)\vee g+g)^+,v\rangle.
\end{align*}
Lemma \ref{abs_difference} implies 
\begin{align*}
S(f_1-\frac 1 2 f_2)=S(f_1)-\frac 12 S(f_2)=S(f_1)-\frac 1 2(S(f_1)\wedge g+S(f_1))^+.
\end{align*}
Thus,
\begin{align*}
I&\leq\langle S(f_1)-\frac 1 2(S(f_1)\wedge g+S(f_1))^++g-\frac 1 2(S(f_1)\vee g+g)^+,v\rangle\\
&\leq \langle S(f_1)+g-\frac 1 2(S(f_1)\wedge g+S(f_1))-\frac 1 2(S(f_1)\vee g+g),v\rangle\\
&=\langle S(f_1)+g-\frac 1 2 (S(f_1)+g+S(f_1)+g),v\rangle\\
&=0,
\end{align*}
where we used $h\leq h^+$ and $h\wedge \tilde h+h\vee \tilde h=h+\tilde h$ for $h,\tilde h\in \K$.\\
Next let us turn to $J$:
\begin{align*}
J{}={}&\langle f_1-\frac 1 2 f_2,-\frac 1 2 f_2\rangle+\langle g-\frac 1 2(S(f_1)\vee g+g)^+,-\frac 1 2(S(f_1)\vee g+g)^+\rangle\\
={}&-\frac 1 2\langle S(f_1),S(f_2)\rangle+\frac 1 4\langle S(f_2),S(f_2)\rangle-\frac 1 2\langle g,(S(f_1)\vee g+g)^+\rangle\\
{}&+\frac 1 4\langle (S(f_1)\vee g+g)^+,(S(f_1)\vee g+g)^+\rangle\\
={}&-\frac 1 2 \langle S(f_1),(S(f_1)\wedge g+S(f_1))^+\rangle+\frac  14 \norm{(S(f_1)\wedge g+S(f_1))^+}^2\\
{}&-\frac 1 2\langle g,(S(f_1)\vee g+g)^+\rangle+\frac 1 4\norm{(S(f_1)\vee g+g)^+}^2.
\end{align*}
Since positive and negative part are orthogonal to each other, we can write
\begin{align*}
\norm{(S(f_1)\wedge g+S(f_1))^+}^2=\langle S(f_1)\wedge g+S(f_1),(S(f_1)\wedge g+S(f_1))^+\rangle
\end{align*}
and likewise for $\norm{(S(f_1)\vee g+g)^+}^2$.\\
Using once again $h\wedge \tilde h+h\vee \tilde h=h+\tilde h$ for $h,\tilde h\in \K$, it follows that
\begin{align*}
4J{}={}&\langle S(f_1)\wedge g -S(f_1),(S(f_1)\wedge g+S(f_1))^+\rangle\\
{}&+\langle S(f_1)\vee g-g,(S(f_1)\vee g+g)^+\rangle\\
={}&\langle g-S(f_1)\vee g,(S(f_1)\wedge g+S(f_1))^+\rangle+\langle S(f_1)\vee g-g,(S(f_1)\vee g+g)^+\rangle\\
={}&\langle S(f_1)\vee g-g,(S(f_1)\vee g+g)^+-(S(f_1)\wedge g+S(f_1))^+\rangle.
\end{align*}
We analyze the factors of the inner product separately.\\
As for the first factor, isotonicity implies that
\begin{align*}
(S(f_1)\vee g+g)^+-(S(f_1)\wedge g+S(f_1))^+\geq (S(f_1)+g)^+-(g+S(f_1))^+=0
\end{align*}
and
\begin{align*}
(S(f_1)\vee g+g)^+-(S(f_1)\wedge g+S(f_1))^+\geq (2g)^+-2S(f_1)\geq 2(g-S(f_1)),
\end{align*}
hence
\begin{align*}
(S(f_1)\vee g+g)^+-(S(f_1)\wedge g+S(f_1))^+\geq 2(g-S(f_1))^+.
\end{align*}
Moreover
\begin{align*}
(S(f_1)\vee g+g)-(S(f_1)\wedge g+S(f_1))&=g-S(f_1)+S(f_1)\vee g-S(f_1)\wedge g\\
&=g-S(f_1)+\abs{g-S(f_1)}\\
&=2(g-S(f_1))^+.
\end{align*}
Let $h,\tilde h\in\K$ such that $\tilde h-h\geq 0,\,\tilde h^+-h^+\geq \tilde h-h$. Then
\begin{align*}
\tilde h^--h^-=(\tilde h^+-h^+)-(\tilde h-h)\geq 0.
\end{align*}
On the other hand, $\tilde h\geq h$ implies by isotonicity $\tilde h^-\leq h^-$. Combining both inequalities yields $\tilde h^-=h^-$ and consequently $\tilde h^+-h^+=\tilde h-h$.\\
Applied to $\tilde h=S(f_1)\vee g+g$ and $h=S(f_1)\wedge g+S(f_1)$ this means that
\begin{align*}
(S(f_1)\vee g+g)^+-(S(f_1)\wedge g+S(f_1))^+=2(g-S(f_1))^+.
\end{align*}
For the other factor in the inner product expression for $4J$ we have:
\begin{align*}
S(f_1)\vee g-g&=\frac 1 2(S(f_1)+g+\abs{S(f_1)-g})-g\\
&=\frac 1 2(S(f_1)-g+\abs{S(f_1)-g})\\
&=(S(f_1)-g)^+\\
&=(g-S(f_1))^-.
\end{align*}
Thus,
\begin{align*}
4J&=\langle S(f_1)\vee g-g,(S(f_1)\vee g+g)^+-(S(f_1)\wedge g+S(f_1))^+\rangle\\
&=2\langle(g-S(f_1))^-,(g-S(f_1))^+\rangle\\
&=0.
\end{align*}
Combining the results for $I$ and $J$ we finally obtain the desired result:
\begin{align*}
\Re\langle (f_1,g)-P(f_1),(u,v)-P(f_1,g)\rangle=I+\Re J\leq 0.&\qedhere
\end{align*}
\qedhere
\end{itemize}
\end{proof}

Of course, the existence of an element $f_2\in\H$ as in (c), (d) is automatically guaranteed  if $S$ is actually a symmetrization. However, for the following corollary we need the above proposition when $S$ is the absolute value on $\K$, which is not necessarily a symmetrization.

\begin{corollary}\label{f_2_unique}
Let $\K^+\subset \K$ be a self-dual positive cone and $S\colon\H\lra\K^+$ a symmetrization. If $f_1\in\H$ and $g\in \K^+$ with $g\leq S(f_1)$, then the element $f_2\in\H$ such that $f_1,f_2$ are $g$-paired is unique.
\end{corollary}
\begin{proof}
This follows from Proposition \ref{domination_invariance} (c) since the projection is unique.
\end{proof}

\begin{corollary}\label{cor_proj_pos_form}
Let $\K^+\subset \K$ be a self-dual isotone projection cone and $\b$ a positive form in $\K$. Let
\begin{align*}
\pi\colon \K^+\oplus\K\lra \K^+\oplus\K^+,\,(h,g)\mapsto \frac 1 2((h\wedge g+h)^+,(h\vee g+g)^+).
\end{align*}
Then $\pi(D(\b)\oplus D(\b))\subset D(\b)\oplus D(\b)$ and
\begin{align*}
(\b\oplus\b)(\pi(h,g),(1-\pi)(h,g))\geq 0
\end{align*}
holds for all $h\in D(\b)^+, g\in D(\b)$.
\end{corollary}
\begin{proof}
Let $(P_t)$ be the semigroup associated with $\b$. By Proposition \ref{invariance_ouhabaz}, $(P_t)$ preserves $\K^+$ and in Example \ref{ex_positive_semigroup_dominating} it was shown that $(P_t)$ is dominated by itself.\\
By Proposition \ref{domination_invariance},
\begin{align*}
C=\{(u,v)\in\K\oplus\K\mid \abs{u}\leq v\}
\end{align*}
is invariant under $(P_t\oplus P_t)$.\\
Since two positive elements of $\K$ are obviously paired (with respect to $\abs\cdot$), we can apply Proposition \ref{domination_invariance} to deduce that the projection onto $C$ satisfies
\begin{align*}
P(h,g)=\frac 1 2 ((h\wedge g+h)^+,(h\vee g+g)^+)=\pi(h,g)
\end{align*}
for all $g\in\K, h\in \K^+$.\\
One more application of Proposition \ref{invariance_ouhabaz} yields $\pi(D(\b)\oplus D(\b))\subset D(\b)\oplus D(\b)$ and
\begin{align*}
(\b\oplus\b)(\pi(h,g),(1-\pi)(h,g))\geq 0
\end{align*}
for all $h\in D(\b)^+, g\in D(\b)$.
\end{proof}

The following two definitions, describing the relation of subspaces under an order structure, will be used in the characterization of domination of semigroups. What we call an ideal is a slight generalization of the notion that is often also found under the name order ideal (the common use is restricted to the case of the absolute value mapping as described below). The second notion, generalized ideal, was originally coined by Ouhabaz (cf. \cite{Ouh96}) under the name ideal, but that collides with the usual terminology in order theory, so we have adapted the usage of \cite{MVV}.

\begin{definition}[Ideal]
Let $\K^+\subset \K$ be a positive cone, $S\colon \H\lra\K^+$ an absolute mapping, and $U,V\subset \H$ subspaces. Then $U$ is called an ideal of $V$ if for $u\in U,v\in V, S(v)\leq S(u)$ implies $v\in U$.
\end{definition}
If $\K^+\subset\K$ is a self-dual isotone projection cone, ideals in $\K$ are understood with respect to the absolute mapping $\abs\cdot\colon\K\lra\K^+$ if not otherwise stated.

\begin{definition}[Generalized ideal]
Let $\K^+\subset\K$ be a positive cone and $S\colon\H\lra\K^+$ a symmetrization. A subspace $U\subset\H$ is called generalized ideal of the subspace $V\subset\K$ if the following properties hold:
\begin{itemize}
\item[(I1)]$S(f)\in V$ for all $f\in U$,
\item[(I2)]For all $f_1\in U$ and $g\in V^+$ such that $g\leq S(f_1)$ there is an $f_2\in U$ such that $f_1,f_2$ are $g$-paired.
\end{itemize}
\end{definition}

This notion is obviously closely related to that of a symmetrization U. However, notice that contrary to the definition of a symmetrization, we only demand the existence of $f_2\in U$ if $g\leq S(f_1)$ here, so that $S$ does not necessarily restrict to a symmetrization $U\lra V^+$.

\begin{remark}
\begin{itemize}
\item Let $U$ be a generalized ideal of $V$. If $f_1\in U$, $g\in V$ such that $g\leq S(f_1)$, there is only one $f_2\in \H$ such that $f_1,f_2$ are $g$-paired by Corollary \ref{f_2_unique}. Then condition (I2) implies that $f_2\in U$.
\item Despite the terminology, not every generalized ideal is an ideal (even in situations when both notions would make sense). As indicated above, this terminology arose due to historical reasons. However, there are some connections between ideals and generalized ideals as discussed for example in \cite{MVV}, Proposition 3.6 and Corollary 3.7.
\end{itemize}
\end{remark}

\begin{example}
Let $X$ be a topological space, $m$ a Borel measure on $X$ and $E$ a Hermitian vector bundle over $X$. Then the subspace $U\subset L^2(X,m;E)$ is a generalized ideal of the subspace $V\subset L^2(X,m)$ if and only if
\begin{itemize}
\item $u\in U$ implies $\abs{u}\in V$,
\item $u\in U, v\in V^+, v\leq \abs{u}$ implies $v\sgn u\in U$.
\end{itemize}
\end{example}
Since $v\leq \abs{u}$ in the second condition, it is irrelevant which value $\sgn u$ has at the zeros of $u$, and we can stick to the usual convention $\sgn u(x)=0$ if $u(x)=0$ instead of taking $\sgn_{\xi}$ from Example \ref{ex_sym_L2}.

\begin{definition}
Let $\K^+\subset \K$ be a positive cone, $S\colon \H\lra \K^+$ a symmetrization, and $\a$ (resp. $\b$) a closed form on $\H$ (resp. $\K$). Then $\a$ is said to be dominated by $\b$ if $D(\a)$ is a generalized ideal of $D(\b)$ and
\begin{align*}
\Re\a(f_1,f_2)\geq \b(S(f_1),S(f_2))
\end{align*}
holds for all $f_1,f_2\in D(\a)$ that are $S(f_2)$-paired.
\end{definition}

Now we can finally give the characterization of domination of semigroups in terms of the associated forms. For comments on the history of this theorem as well as on the proof, see the remarks below.

\begin{theorem}\label{thm_char_domination}
Let $\K^+\subset\K$ be a self-dual positive cone, $S\colon\H\lra\K^+$ a symmetrization, $A$ (resp. $B$) a self-adjoint operator on $\H$ (resp. $\K)$ with lower bound $-\lambda$, and $\a$ (resp. $\b)$ the associated form. Then the following assertions are equivalent:
\begin{itemize}
\item[(i)]The semigroup $(e^{-tA})_{t\geq 0}$ is dominated by $(e^{-tB})_{t\geq 0}$.
\item[(ii)]The resolvent $((A+\alpha)^{-1})_{\alpha>\lambda}$ is dominated by $((B+\alpha)^{-1})_{\alpha>\lambda}$.
\end{itemize}
Both assertions imply
\begin{itemize}
\item[(iii)]The form $\a$ is dominated by $\b$.

\end{itemize}
If $\K^+$ is a self-dual isotone projection cone, the assertions (i), (ii) and (iii) are equivalent.
\end{theorem}
\begin{proof}
(i)$\implies$(ii): Let $f\in\H,\alpha>\lambda$. The resolvent is given as the Laplace transform of the semigroup:
\begin{align*}
(A+\alpha)^{-1}f=\int_0^\infty e^{-t\alpha}e^{-tA}f\,dt
\end{align*}
and likewise for $B$.\\
Then we have for all $g\in\K^+$:
\begin{align*}
\langle S((A+\alpha)^{-1}f),g\rangle&=\left\langle S\left(\int_0^\infty e^{-t\alpha} e^{-tA}f\,dt\right),g\right\rangle\\
&\leq\left\langle \int_0^\infty e^{-t\alpha}S(e^{-tA} f)\,dt,g\right\rangle\\
&=\int_0^\infty e^{-t\alpha} \langle S(e^{-tA} f),g\rangle\,dt\\
&\leq \int_0^\infty e^{-t\alpha} \langle e^{-tB} Sf,g\rangle\,dt\\
&=\left\langle \int_0^\infty e^{-t\alpha} e^{-tB} Sf\,dt,g\right\rangle\\
&=\langle (B+\alpha)^{-1}Sf,g\rangle.
\end{align*}
Observe that the first inequality holds due to homogeneity and the triangle inequality for finite Riemann sums. By continuity of $S$, it follows for the integral.\\
Hence, $((\alpha+A)^{-1})_{\alpha>\lambda}$ is dominated by $((\alpha+B)^{-1})_{\alpha>\lambda}$.

(ii)$\implies$(i): Let $f\in \H, t\geq 0$. The semigroup can be derived from the resolvent via 
\begin{align*}
e^{-tA}f=\lim_{n\to\infty}\left(\frac n t\right)^n\left(A+\frac n t\right)^{-n}f
\end{align*}
and likewise for $B$.\\
Since $(e^{-tB})$ is positivity preserving, $(B+\alpha)^{-1}$ preserves $\K^+$ for all $\alpha>\lambda$ by \ref{invariance_ouhabaz}. By Lemma \ref{domination_sums}, $((A+\alpha)^{-n})_{\alpha>\lambda}$ is dominated by $((B+\alpha)^{-n})_{\alpha>\lambda}$ for all $n\in\IN$.\\
Thus, using the elementary properties of $S$, we have for all $g\in\K^+$:
\begin{align*}
\langle S(e^{-tA} f),g\rangle&=\lim_{n\to\infty}\left(\frac n t\right)^n\left\langle S\left(\left(A+\frac n t\right)^{-n}f\right),g\right\rangle\\
&\leq \lim_{n\to\infty}\left(\frac n t\right)^n\left\langle \left(B+\frac n t\right)^{-n}S(f),g\right\rangle\\
&=\langle e^{-tB} S(f),g\rangle.
\end{align*}
Hence, $(e^{-tA})$ is dominated by $(e^{-tB})$.\\
Next we do some preparatory work for (i)$\implies$(iii) and (iii)$\implies$(i) in the isotone case.\\
Define $C=\{(u,v)\in\H\oplus\K\mid S(u)\leq v\}$ and $W_t=e^{-tA}\oplus e^{-tB}\colon\H\oplus \K\lra\H\oplus \K$ for $t\geq 0$.\\
By Proposition \ref{domination_invariance} (b), (i) is equivalent to $W_t C\subset C$ for all $t\geq 0$.\\
The form $\tau$ associated to $(W_t)$ is given by
\begin{align*}
D(\tau)&=\{(u,v)\in\H\oplus\K\mid \lim_{t\downto 0}\frac 1 t\langle (u,v)-W_t(u,v),(u,v)\rangle<\infty\}\\
&=\{(u,v)\in\H\oplus\K\mid \lim_{t\downto 0}\frac 1 t(\langle u-e^{-tA}u,u\rangle+\langle v-e^{-tB}v,v\rangle)<\infty\}\\
&=D(\a)\oplus D(\b),\\
\tau((u,v))&=\lim_{t\downto 0}\frac 1 t(\langle u-e^{-tA}u,u\rangle+\langle v-e^{-tB}v,v\rangle)\\
&=\a(u)+\b(v).
\end{align*}
By Proposition \ref{invariance_ouhabaz}, $C$ is invariant under $(W_t)$ if and only if $P_C D(\tau)\subset D(\tau)$ and $\Re \tau(P_C(f,g),(f,g)-P_C(f,g))\geq 0$ for all $(f,g)\in D(\a)\oplus D(\b)$.

(i)$\implies$(iii): By Proposition \ref{domination_invariance} (b), the projection $P_C$ satisfies
\begin{align*}
P_C(f_1,g)=\frac 1 2(f_1+f_2,S(f_1)+g)
\end{align*}
for all $f_1,f_2\in\H$, $g\in\K^+$ such that $f_1,f_2$ are $g$-paired and $g\leq S(f_1)$.\\
Now assume additionally that $f_1\in D(\a),\,g\in D(\b)$. If $C$ is invariant under $(W_t)$, then 
\begin{align*}
P_C(f_1,g)=\frac 1 2(f_1+f_2,S(f_1)+g)\in D(\a)\oplus D(\b),
\end{align*}
hence $f_2\in D(\a)$ and $S(f_1)\in D(\b)$. Thus, $D(\a)$ is a generalized ideal of $D(\b)$.\\
Let $f_1,f_2\in D(\a)$ be $S(f_2)$-paired. Then $S(f_1),S(f_2)\in D(\b)$ and 
\begin{align*}
\Re \a(f_1,f_2)&=\lim_{t\downto 0}\frac 1 t\Re\langle f_1-e^{-tA}f_1,f_2\rangle\\
&=\lim_{t\downto 0}\frac 1 t(\langle S(f_1),S(f_2)\rangle-\Re\langle e^{-tA}f_1,f_2\rangle)\\
&\geq\lim_{t\downto 0}\frac 1 t (\langle S(f_1),S(f_2)\rangle-\langle e^{-tB}S(f_1),S(f_2)\rangle)\\
&=\lim_{t\downto 0}\frac 1 t\langle Sf_1-e^{-tB}S(f_1),S(f_2)\rangle\\
&=\b(S(f_1),S(f_2)).
\end{align*}
For the remainder of the proof we assume that $\K^+$ is a self-dual isotone projection cone.

(iii)$\implies$(i): First we show that $P_C D(\tau)\subset D(\tau)$. For that purpose let $(f_1,g)\in D(\tau)$. By Proposition \ref{domination_invariance} (d), the projection $P_C(f_1,g)$ is given by
\begin{align*}
P_C(f_1,g)=\frac 1 2(f_2, (S(f_1)\vee g+g)^+),
\end{align*}
where $f_2\in \H$ such that $f_1,f_2$ are $(S(f_1)\wedge g+S(f_1))^+$-paired.\\
By isotonicity, 
\begin{align*}
S(f_2)=(S(f_1)\wedge g+S(f_1))^+\leq (2S(f_1))^+=2 S(f_1).
\end{align*} 
Since $D(\a)$ is generalized ideal of $D(\b)$, this inequality implies $f_2\in D(\a)$.\\
Furthermore, $S(f_1)\in D(\b)$ once again since $D(\a)$ is a generalized ideal in $D(\b)$ and therefore $\frac 1 2 (S(f_1)\vee g+g)^+\in D(\b)$ since $D(\b)$ is a sublattice of $\K$ by Lemma \ref{pos_form_lattice}. Hence, $P_C(f_1,g)\in D(\tau)$.\\
Let $g_2=\frac 1 2(S(f_1)\vee g+g)^+$. By (iii) and Lemma \ref{abs_difference} we have
\begin{align*}
\Re \tau(P_C(f_1,g),(1-P_C)(f_1,g))&=\Re\a\left(\frac 1 2 f_2,f_1-\frac 1 2 f_2\right)+\b\left(g_2,g- g_2\right)\\
&\geq \b\left(\frac 1 2 S(f_2), S(f_1-\frac 1 2 f_2)\right)+\b(g_2,g-g_2)\\
&= \b\left(\frac 1 2 S(f_2),S(f_1)-\frac 1 2S(f_2)\right)+\b\left( g_2,g- g_2\right)\\
&=(\b\oplus\b)\left(\frac 1 2(S(f_2),g_2),(S(f_1),g)-(S(f_2),g_2)\right).
\end{align*}
Now Corollary \ref{cor_proj_pos_form} implies $(S(f_2),g_2)=\pi(S(f_1),g)$ and
\begin{align*}
(\b\oplus\b)(\pi(S(f_1),g),(1-\pi)(S(f_1),g))\geq 0.
\end{align*}
Therefore,
\begin{align*}
\Re\tau(P_C(f_1,g),(1-P_C)(f_1,g))\geq 0.
\end{align*}
In the light of our preparatory work this means that $(e^{-tA})$ is dominated by $(e^{-tB})$. 
\end{proof}

\begin{remark}
\begin{itemize}
\item A first version of this theorem was given independently by Simon (\cite{Sim77}, Thm. 5.1) for operators on $L^2$-spaces and by Hess, Schrader, Uhlenbrock (\cite{HSU}) in the setting of symmetrizations between abstract Hilbert spaces. Both did not offer a characterization purely in terms of forms, but the following inequality:
\begin{align*}
\Re \langle \overline{g\sgn f} ,Af\rangle\geq \b(\abs{f},g)
\end{align*}
for $f\in D(A),\,g\in D(\b)^+$.
\item The characterization in terms of the associated forms was first given by Ouhabaz (cf. \cite{Ouh96}) for semigroups on $L^2$-spaces (vector-valued for $\H$).
\item Although phrased in the terminology of \cite{HSU}, our proof is closer related to those in \cite{Ouh96, MVV} in that we take Proposition \ref{invariance_ouhabaz} as a main ingredient. This approach has also the advantage that (iii) can be phrased in terms of forms and does not involve the domain of the generator. This is not only essential for our application, but also seemed conceptually better fitting.
\item The work of Manavi, Vogt and Voigt (\cite{MVV}) is concerned with some further ramifications of this theorem in the case of $L^2$-spaces, allowing not necessarily densely defined, sectorial forms and giving criteria on cores. We believe that those carry over to our more abstract setting, however, that was not the focus of this article and we will not need such criteria later.
\item If $\a$ is dominated by $\b$, $f_1\in D(\a)$ and $g\in D(\b)_+$ with $g\leq S(f_1)$, there is an $f_2\in D(\a)$ such that $f_1, f_2$ are $g$-paired since $D(\a)$ is a generalized ideal of $D(\b)$. The condition $g\leq S(f_1)$ cannot be dropped, as the following example shows:\\
Let $\E$ be the standard energy form on $\IR^n$, that is,
\begin{align*}
D(\E)=H^1(\IR^n), \E(u)=\int_{\IR^n}\abs{\nabla u}^2\,dx.
\end{align*}
Then $\E$ is positive, hence dominated by itself.\\
Let $f_1\in C_c^\infty(\IR^n)$ be such that $\supp f_1\subset[-1,1]$ and $f_1|_{[0,1]}\geq 0,\,f_1|_{[-1,0]}\leq 0$. Let $g\in C_c^\infty(\IR^n)$, $g\geq 0$, $g|_{[-2,2]}=1$. If $f_1$ and $f_2$ are $g$-paired, then $f_2(x)=g(x)\sgn f_1(x)=\sgn (x)$ for all $x\in [-1,1]$. Hence, $f_2\notin H^1(\IR^n)$.
\item It is amazing that a stronger assumption on $\K^+$ is needed for the implication (iii)$\implies$(i) while the theorem in \cite{HSU} works without further assumption on $\K^+$. However, it is obvious that our proof strategy strongly relies on the fact that $\K^+$ is an isotone projection cone and $\K$ therefore a Riesz space. The proof in \cite{HSU} also does not carry over to our situation as far as we can see.
\end{itemize}
\end{remark}

For convenience we reformulate the above theorem for the case of $L^2$-spaces.
\begin{corollary}\label{cor_domination}
Let $X$ be a topological space, $m$ a Borel measure on $X$ and $E$ a Hermitian vector bundle over $X$, $A$ (resp. $B$) a lower semibounded, self-adjoint operator on $L^2(X,m;E)$ (resp. $L^2(X,m)$) and $\a$ (resp. $\b$) the associated form. Then the following are equivalent:
\begin{itemize}
\item[(i)]The semigroup $(e^{-tA})$ is dominated by $(e^{-tB})$.
\item[(ii)]The domain $D(\a)$ is a generalized ideal of $D(\b)$ and 
\begin{align*}
\Re\a(u,\tilde u)\geq \b(\abs{u},\abs{\tilde u})
\end{align*}
holds for all $u,\tilde u\in D(\a)$ satisfying $\langle u(x), \tilde u(x)\rangle_x=\abs{u(x)}_x\abs{\tilde u(x)}_x$ for almost all $x\in X$.
\end{itemize}
\end{corollary}
\begin{proof}
It only remains to prove that $u,\tilde u\in L^2(X,m;E)$ are paired if and only if $\langle u(x),\tilde u(x)\rangle_x=\abs{u(x)}_x\abs{\tilde u(x)}_x$ holds for almost all $x\in X$.\\
If $\langle u(x),\tilde u(x)\rangle_x=\abs{u(x)}_x\abs{\tilde u(x)}_x$ holds for almost all $x\in X$, then
\begin{align*}
\langle u,\tilde u\rangle_{L^2(X,m;E)}&=\int_X\langle u(x),\tilde u(x)\rangle_x\,dm(x)\\
&=\int_X\abs{u(x)}_x\abs{\tilde u(x)}_x\,dm(x)\\
&=\langle \abs{u},\abs{\tilde u}\rangle_{L^2(X,m)}.
\end{align*}
Hence $u$ and $\tilde u$ are paired.\\
Conversely assume that $u$ and $\tilde u$ are paired. Then
\begin{align*}
0=\langle \abs{u},\abs{\tilde u}\rangle_{L^2(X,m)}-\langle u,\tilde u\rangle_{L^2(X,m;E)}=\int_X(\abs{u(x)}_x\abs{\tilde u(x)}_x-\langle u(x),\tilde u(x)\rangle_x)\,dm(x)
\end{align*}
and the integrand is positive by Cauchy-Schwarz inequality. Therefore it must be zero almost everywhere.
\end{proof}

\begin{example}
Let $(X,\B,m)$ be a measure space and $(Q_t)$ a positivity preserving semigroup on $L^2(X,m)$ with associated form $\b$. In Example \ref{ex_positive_semigroup_dominating} it was remarked that $(Q_t)$ is dominated by itself, hence $D(\b)$ is a generalized ideal in itself and
\begin{align*}
\Re \b(u,v)\geq\b(\abs{u},\abs{v})
\end{align*}
holds for all $u,v\in D(\b)$ satisfying $u\overline v=\abs{u}\abs{v}$.\\
Now let $V\colon X\lra [0,\infty)$ be measurable and define the form $\a$ via
\begin{align*}
D(\a)=\{u\in D(\b)\mid V^{\frac 1 2}u\in L^2(X,m)\},\,\a(u)=\b(u)+\int_X V\abs{u}^2\,dm.
\end{align*}
Then $\a$ is dominated by $\b$: If $u\in D(\a)\subset D(\b)$, then $\abs{u}\in D(\b)$ since $\b$ is positive. If $u\in D(\a), v\in D(\b)^+$ such that $v\leq \abs{u}$, then $v\sgn u\in D(\b)$ since $D(\b)$ is a generalized ideal in itself, and \begin{align*}
\int_X V\abs{v\sgn u}^2\,dm\leq\int_X V\abs{u}^2\,dm<\infty,
\end{align*}
hence $v\sgn u\in D(\a)$.\\
Moreover, for all $u,v\in D(\a)$ satisfying $u\overline v=\abs{u}\abs{v}$ we have
\begin{align*}
\Re \a(u,v)=\Re\b(u,v)+\int_X V u\overline v\,dm\geq \b(\abs{u},\abs{v})+\int_X V\abs{u}\abs{v}\,dm=\a(\abs{u},\abs{v}).
\end{align*}
\end{example}

In the light of Example \ref{ex_positive_semigroup_dominating}, every positive form dominates a form, namely itself. Combined with the following corollary this gives a full characterization of the forms that can occur as dominating forms (if the positive cone is a self-dual, isotone projection cone).
\begin{corollary}\label{dominating_form_pos}
Let $\K^+\subset\K$ be a self-dual, isotone projection cone, $S\colon\H\lra\K^+$ a symmetrization and $A$ (resp. $B$) a lower semibounded, self-adjoint operator on $\H$ (resp. $\K)$, and $\a$ (resp. $\b)$ the associated form.\\
If $\a$ is dominated by $\b$, then $\b$ is a positive form.
\end{corollary}
\begin{proof}
By the characterization of domination (Theorem \ref{thm_char_domination}), $(e^{-tA})$ is dominated by $(e^{-tB})$. Then Lemma \ref{domination_sums} implies that $(e^{-tB})$ leaves $\K^+$ invariant. By Proposition \ref{invariance_ouhabaz}, this is equivalent to the positivity of $\b$.
\end{proof}

\section{A criterion for form uniqueness}\label{chap_uniqueness_criterion}

Having set up the stage in the last chapter, we can now almost immediately turn to the main theorem of this article, which allows to transfer form uniqueness of a dominating form to that of the dominated form.

The main part of the work is done in the following technical lemma, which might also be of interest in other situations.

\begin{lemma}\label{approx_pos}
Let $\K$ be a real Hilbert space, $\K^+\subset\K$ a self-dual isotone projection cone, $(\b,D(\b))$ a closed, positive form on $\K$ and $D_\b\subset D(\b)$ a dense ideal. If $v\in D(\b)^+$, then there is a sequence $(v_n)$ in $D_\b$ such that $0\leq v_n\leq v$ and $\norm{v_n-v}_\b\to 0$.
\end{lemma}
\begin{proof}
Let $v\in D(\b)^+$. Since $D_\b\subset D(\b)$ is dense, there is a sequence $\tilde v_n$ in $D_\b$ such that $\norm{\tilde v_n-v}_\b\to 0$. By Lemma \ref{pos_form_lattice}, $D(\b)$ is a sublattice of $\K$, hence $\tilde v_n^+\wedge v\in D(\b)$. 
From the inequality $0\leq \tilde v_n^+\wedge v\leq v_n^+\leq \abs{\tilde v_n}$ it follows that $\tilde v_n^+\wedge v\in D_\b$ since $D_\b\subset D(\b)$ is an ideal.\\
Let $-\lambda<0$ be a lower bound for $\b$. Then Lemma \ref{pos_form_lattice} implies that
\begin{align*}
\norm{\tilde v_n^+\wedge v}_{\b}^2&=\b_{1+\lambda}(\tilde v_n^+\wedge v)\\
&\leq \b_{1+\lambda}(\tilde v_n^+)+\b_{1+\lambda}(v)\\
&\leq \b_{1+\lambda}(\tilde v_n)+\b_{1+\lambda}(v)\\
&=\norm{\tilde v_n}_{\b}^2+\norm{v}_{\b}^2
\end{align*}
Since $(\tilde v_n)$ is convergent in $(D(\b),\langle\cdot,\cdot\rangle_\b)$, it is in particular bounded, and the above  inequality shows that $(\tilde v_n^+\wedge v)$ is bounded as well.\\
By the Banach-Saks Theorem (cf. \cite{Wer}, Satz V.3.8) there is a subsequence $(\tilde v_{n_k})$ and an element $\tilde v\in D(\b)$ such that
\begin{align*}
v_N:=\frac 1 N\sum_{k=1}^N \tilde v_{n_k}^+\wedge v\overset{\norm\cdot_\b}{\to}\tilde v,\,N\to\infty.
\end{align*}
Obviously, $v_N\in D_\b$ and $0\leq v_N\leq v$ for all $N\in \IN$.\\
Moreover, by Lemma \ref{basic_Riesz} we have
\begin{align*}
\abs{v-\tilde v_n}=(v\vee \tilde v_n)-(v\wedge \tilde v_n)\geq v-v\wedge \tilde v_n
\end{align*}
and
\begin{align*}
v\wedge \tilde v_n^+=v\wedge(\tilde v_n\vee 0)=(v\wedge \tilde v_n)\vee (v\wedge 0)\geq v\wedge \tilde v_n,
\end{align*}
hence
\begin{align*}
\abs{v-\tilde v_n}\geq v-v\wedge \tilde v_n^+\geq 0.
\end{align*}
By Lemma \ref{norm_Riesz_monotone}, this inequality implies $\norm{v-v\wedge \tilde v_n^+}_\K\leq \norm{v-\tilde v_n}_\K\to 0$ and therefore also $\norm{v- v_N}_\K\to 0$. Thus, $\tilde v=v$.
\end{proof}

\begin{theorem}\label{uniqueness_forms}
Let $\H$ be a Hilbert space, $\K$ a real Hilbert space, $\K^+\subset \K$ a self-dual isotone projection cone, and $S\colon \H\lra\K^+$ a symmetrization.\\
Let $(\a,D(\a))$ be a closed form in $\H$, $(\b,D(\b))$ a closed form in $\K$, and $D_\a\subset D(\a), D_{\b}\subset D(\b)$ ideals such that the following conditions hold:
\begin{itemize}
\item $\a$ is dominated by $\b$
\item $D_\b^+\cap S(D(\a))\subset S(D_\a)$
\end{itemize}
If $D_\b$ is a form core for $\b$, then $D_\a$ is a form core for $\a$.
\end{theorem}
\begin{proof}
Let $-\lambda<0$ be a common lower bound for $\a$ and $\b$. As $\a$ is closed, $D(\a)$ is a Hilbert space with the inner product $\langle\cdot,\cdot\rangle_\a=(1+\lambda)\langle\cdot,\cdot\rangle_{\H}+\a(\cdot,\cdot)$ and analogously for $\b$.\\
We show that $D_\a\subset D(\a)$ is dense with respect to $\norm\cdot_\a$ by proving that $D_\a^\perp=\{0\}$ in $(D(\a),\langle\cdot,\cdot\rangle_\a)$. For this purpose, let $h\in D(\a)$ such that
\begin{align*}
0=\langle h,u\rangle_\a=(1+\lambda)\langle h,u\rangle+\a(h,u)
\end{align*}
for all $u\in D_\a$.\\
Next take $v\in D_\b^+$ such that $v\leq S(h)$. Since $\a$ is dominated by $\b$, $D(\a)$ is a generalized ideal of $D(\b)$. Hence there is an $\tilde h\in D(\a)$ such that $h,\tilde h$ are $v$-paired. In particular, $S(\tilde h)=v\in D_\b^+\cap S(D(\a))\subset S(D_\a)$, and since $D_\a$ is an ideal in $D(\a)$, $\tilde h\in D_\a$.\\
Moreover, since $\a$ is dominated by $\b$, $S(h)\in D(\b)$ and
\begin{align*}
0=(1+\lambda)\langle h,\tilde h\rangle+\Re\a(h,\tilde h)\geq (1+\lambda)\langle S(h),v)\rangle+b(S(h),v).\tag{$\ast$}\label{ineq}
\end{align*}
By Lemma \ref{dominating_form_pos}, the form $\b$ is positive. An application of Lemma \ref{approx_pos} yields a sequence $(v_n)$ in $D_\b$ such that $0\leq v_n\leq S(h)$ and $\norm{v_n-S(h)}_\b\to 0$.\\
Applying inequality (\ref{ineq}) to $v= v_N$ we obtain
\begin{align*}
0\geq (1+\lambda)\langle S(h), v_N\rangle+\b(S(h), v_N)=\langle S(h), v_N\rangle_\b\to \norm{S(h)}_\b^2.
\end{align*}
Hence $S(h)=0$ and therefore also $h=0$. Thus, $D_\a^\perp=\{0\}$.
\end{proof}

\begin{remark}
\begin{itemize}
\item In applications, the situation will often be as follows: We are given forms $\a_0$ on $D_\a$, $\b_0$ on $D_\b$ (usually not closed) and minimal extensions $\a_{\text{min}}$, $\b_{\text{min}}$ (the closures of $\a_0$, $\b_0$) and maximal extensions $\a_{\text{max}}$, $\b_{\text{max}}$.\\
If $\b_{\text{min}}=\b_{\text{max}}$, then the theorem yields $\a_{\text{min}}=\a_{\text{max}}$.\\
This is discussed in detail in Chapter \ref{Applications}.
\item If $\H=\K$ and $S=\abs\cdot$, then $D_\a=D_\b$ satisfies the condition $D_ \b ^+\subset S(D_\a)$.
\item In the light of Theorem \ref{thm_char_domination}, the condition that $\a$ is dominated by $\b$ can also be phrased in terms of the associated semigroups or resolvents.
\item In the case of $L^2$-spaces, the application of the Banach-Saks Theorem in the proof can be replaced by the fact that the $L^2$-convergent sequence $(v_n)$ has a pointwise convergent subsequence.
\end{itemize}
\end{remark}

We conclude this section with a corollary that matches the concrete situation of our applications in the next chapter.

\begin{corollary}\label{corollary_uniqueness_forms}
Let $X$ be a topological space and $\mu$ a Borel measure on $X$. Denote by $L^\infty_c(X,\mu)$ the space of essentially bounded functions that vanish outside a compact set.\\
Let $E\lra X$ be a Hermitian vector bundle over $X$, and denote by $L^\infty_c(X,\mu;E)$ the space of essentially bounded sections in $E$ that vanish outside a compact set.\\
Assume that $(\b,D(\b))$ is a closed form in $L^2(X,\mu)$ and $(\a,D(\a))$ a closed form in $L^2(X,\mu;E)$ that is dominated by $\b$.\\
If $D(\b)\cap L^\infty_c(X,\mu)$ is a form core for $\b$, then $D(\a)\cap L^\infty_c(X,\mu;E)$ is a form core for $\a$.
\end{corollary}
\begin{proof}
We will apply Theorem \ref{uniqueness_forms} to $D_\a=L^\infty_c(X,\mu;E)\cap D(\a)$ and $D_\b=L^\infty_c(X,\mu)\cap D(\b)$. It is obvious that these are ideals in $D(\a)$ and $D(\b)$ respectively.\\
Now let $g\in D_\b^+\cap \abs{D(\a)}$. Then there is an $f\in D(\a)$ such that $\abs{f}=g\in L^\infty_c(X,\mu)$. Thus, $f\in L^\infty_c(X,\mu;E)\cap D(\a)$ and $g=\abs{f}\in \abs{D_\a}$. 
\end{proof}

\begin{remark}
\begin{itemize}
\item We tacitly assumed that $L^2(X,\mu)$ is the space of real-valued $L^2$-functions in order to apply Theorem \ref{uniqueness_forms}, whereas $L^2(X,\mu;E)$ may be viewed either as real or as complex. We will adopt this convention also for the applications of this corollary in the next chapter.
\item If $(\b,D(\b))$ is a regular Dirichlet form, $L^\infty_c(X,\mu)\cap D(\b)$ is a form core for $\b$. Indeed, $C_c(X)\cap D(\b)\subset L^\infty_c(X,\mu)\cap D(\b)$ is dense in $D(\b)$ by definition. Note, however, that we do not use the second Beurling-Deny criterion in our reasoning.
\item In our applications, we will only encounter trivial vector bundles. Thus, $L^2(X,\mu;E)$ can be identified with $L^2(X,\mu;\IC^n)$ via the trivialization. However, this is not necessarily the case in other possible applications, for example if $a$ is the form associated with the Hodge-de Rham Laplacian on $p$-forms on a manifold.
\end{itemize}
\end{remark}

In the smooth case, one is usually interested in the closure of the form defined on smooth functions (sections) as minimal form. We make the following definition adapted to this situation.

\begin{definition}
Let $M$ be a Riemannian manifold and $E\lra M$ a smooth Hermitian vector bundle. A form $\a$ on $L^2(M;E)$ is called \emph{smoothly inner regular} if $\Gamma_c^\infty(M;E)\cap D(\a)$ is dense in $D(\a)\cap L^\infty_c(M;E)$ with respect to $\norm\cdot_\a$.
\end{definition}

From the definition of smooth inner regularity and the above corollary, the following corollary can easily be deduced.

\begin{corollary}\label{uniqueness_smooth}
Let $M$ be a Riemannian manifold and $E\lra M$ a smooth Hermitian vector bundle. Let $\b$ be a closed form on $L^2(M)$ and $\a$ a closed, smoothly inner regular form on $L^2(M;E)$ that is dominated by $\b$.\\
If $C_c^\infty(M)\cap D(\b)$ is a form core for $\b$, then $\Gamma_c^\infty(M;E)\cap D(\a)$ is a form core for $\a$.
\end{corollary}

\section{Applications}\label{Applications}

\subsection{Magnetic Schrödinger forms on graphs}\label{Magnetic Schrödinger operators on graphs}
In this section we will study discrete analogs of the Laplacian respectively magnetic Schrödinger operators in Euclidean space. Analysis on graphs has been an active field of research in recent years and uniqueness of extensions of operators respectively forms on graphs have been intensively studied. We just point to \cite{HKLW}, \cite{HKMW} for non-magnetic forms and the recent series of works by Milatovic and Truc (\cite{MT14},\cite{MT15}) for magnetic forms as a few examples; this list is by no means meant to be complete.

Compared to the Euclidean case, the discrete setting allows more clarity in the presentation as some mere technical complications do not appear. In particular, Corollary \ref{corollary_uniqueness_forms} can be applied directly since $L_c^\infty(X)$ and $C_c(X)$ coincide for discrete spaces.

We will start with some basic definitions, including those of magnetic Schrö\-dinger forms on graphs (Definitions \ref{def_magnetic_graph_Neumann} and \ref{def_magnetic_graph_Dirichlet}), essentially following \cite{KL10},\cite{KL12} regarding graphs and Dirichlet forms over discrete spaces and \cite{MT15} regarding vector bundles over graphs and magnetic Schrödinger operators. Then we show that the form with magnetic field is dominated by the form without magnetic field (Proposition \ref{domination_magnetic_forms}) before we finally give the uniqueness result (Theorem \ref{form_uniqueness_graphs}) and discuss some examples.

\begin{definition}[Weighted graph]
A weighted graph $(X,b,c,m)$ consists of an (at most) countable set $X$, an edge weight $b\colon X\times X\lra [0,\infty)$, a killing term $c\colon X\lra[0,\infty)$ and a measure $m\colon X\lra(0,\infty)$ subject to the following conditions:
\begin{itemize}
\item[(b1)]$b(x,x)=0$,
\item[(b2)]$b(x,y)=b(y,x)$,
\item[(b3)]$\sum_{z\in X}b(x,z)<\infty$
\end{itemize}
for all $x,y\in X$.
\end{definition}
Observe that we do not assume our graphs to be locally finite, that is, $\{y\in X\mid b(x,y)>0\}$ may be infinite as long the edge weights are still summable.\\
We shall regard $X$ as a discrete topological space. Consequently, $C_c(X)$ is the space of functions on $X$ with finite support.\\
We regard $m$ as a measure on $\mathcal{P}(X)$ via
\begin{align*}
m(A):=\sum_{x\in A} m(x),\;A\subset X,
\end{align*}
and denote the corresponding $L^2$-space by $\ell^2(X,m)$.

\begin{definition}[Hermitian vector bundle]
A \emph{Hermitian vector bundle} over a discrete set $X$ is a family $F=(F_x,\langle\cdot\,,\cdot\rangle_x)_{x\in X}$ of finite-dimensional Hilbert spaces together with a unitary connection $\Phi_{x,y}\colon F_y\lra F_x$ for all $x,y\in X$ such that $\Phi_{x,y}=\Phi_{y,x}^{-1}$.\\
For a Hermitian vector bundle $F$ over $X$ we will denote by
\begin{align*}
\Gamma(X;F)&=\{u\colon X\lra\prod_{x\in X}F_x\mid u(x)\in F_x\},\\
\Gamma_c(X;F)&=\{u\in\Gamma(X;F)\mid \supp u\text{ finite}\},\\
\ell^2(X,m;F)&=\{u\in\Gamma(X;F)\mid\sum_{x\in X}\langle u(x),u(x)\rangle_x m(x)<\infty\}
\end{align*}
the space of all sections, the space of all sections with compact support and the space of all $L^2$-sections. The latter becomes a Hilbert space equipped with the inner product
\begin{align*}
\langle \cdot,\cdot\rangle_{\ell^2(X,m;F)}\colon\ell^2(X,m;F)\times\ell^2(X,m;F)\lra\IC,\,(u,v)\mapsto\sum_{x\in X}\langle u(x),v(x)\rangle_x m(x).
\end{align*}
A bundle endomorphism $W$ of a Hermitian vector bundle $F$ is a family of linear maps $(W(x)\colon F_x\lra F_x)_{x\in X}$.
\end{definition}

For the remainder of the section, $(X,b,c,m)$ is a weighted graph, $F$ a Hermitian vector bundle over $X$ with unitary connection $\Phi$ and $W$ a bundle endomorphism of $F$ that is pointwise positive, that is, $\langle W(x)v,v\rangle_x\geq 0$ for all $x\in X,\,v\in F_x$.

Now we can define the basic object of our interest, the magnetic Schrödinger form (with Dirichlet and Neumann boundary conditions).

\begin{definition}[Magnetic form with Neumann boundary conditions]\label{def_magnetic_graph_Neumann}
For $u\in \Gamma(X;F)$ let
\begin{align*}
\tilde Q_{\Phi,b,W}(u)=\frac 1 2\sum_{x,y}b(x,y)\abs{u(x)-\Phi_{x,y}u(y)}_x^2+\sum_{x}\langle W(x)u(x),u(x)\rangle_x\in[0,\infty].
\end{align*}
The magnetic Schrödinger form with Neumann boundary conditions is defined via
\begin{align*}
D(Q^{(N)}_{\Phi,b,W})&=\{u\in\ell^2(X,m)\mid\tilde Q_{\Phi,b,W}(u)<\infty\},\\
Q^{(N)}_{\Phi,b,W}(u)&=\tilde Q_{\Phi,b,W}(u).
\end{align*}
\end{definition}

To abridge notation, we will write $\norm{\cdot}_{\Phi,b,W}$ for the form norm of $Q^{(N)}_{\Phi,b,W}$. In the next lemma we show that $Q^{(N)}_{\Phi,b,W}$ is closed. This fact will then be used to define the magnetic Schrödinger form with Dirichlet boundary conditions.

\begin{lemma}
The form $Q^{(N)}_{\Phi,b,W}$ is closed.
\end{lemma}
\begin{proof}
Let $(u_n)$ be a Cauchy sequence in $D(Q^{(N)}_{\Phi,b,W})$ and $u_n\to u$ in $\ell^2(X,m;F)$. By Fatou's lemma,
\begin{align*}
\tilde Q_{\Phi,b,W}(u)&=\frac 1 2\sum_{x,y}\lim_{n\to \infty}b(x,y)\abs{u_n(x)-\Phi_{x,y}u_n(y)}^2+\sum_{x}\lim_{n\to\infty}\langle W(x)u_n(x),u_n(x)\rangle_x\\
&\leq\liminf_{n\to\infty}\left(\frac 1 2\sum_{x,y}b(x,y)\abs{u(x)-\Phi_{x,y}u(y)}_x^2+\sum_{x}\langle W(x)u(x),u(x)\rangle_x\right)\\
&=\liminf_{n\to\infty}\tilde Q_{\Phi,b,W}(u_n)\\
&<\infty.
\end{align*}
Thus, $u\in D(Q^{(N)}_{\Phi,b,W})$.\\
By the same argument,
\begin{align*}
Q^{(N)}_{\Phi,b,W}(u-u_n)&=Q^{(N)}_{\Phi,b,W}(\lim_{m\to\infty}(u_m-u_n))\\
&\leq\liminf_{m\to\infty}Q^{(N)}_{\Phi,b,W}(u_m-u_n)\\
&\to 0,\,n\to\infty.
\end{align*}
Therefore, $\norm{u_n-u}_{\Phi,b,W}\to 0$.
\end{proof}

\begin{definition}[Magnetic form with Dirichlet boundary conditions]\label{def_magnetic_graph_Dirichlet}
The magnetic Schrödinger form with Dirichlet boundary conditions $Q^{(D)}_{\Phi,b,W}$ is the closure of the restriction of $Q^{(N)}_{\Phi,b,W}$ to $C_c(X)$.
\end{definition}

If $F_x=\IC$ endowed with the standard inner product and $\Phi_{x,y}=1$ for all $x,y\in X$, we will suppress $\Phi$ in the index and simply write $Q^{(D)}_{b,W}$ (resp. $Q^{(N)}_{b,W})$. We may also drop other indices if they are clear from the context.\\
The interest in these forms is particularly a result of the fact that $Q^{(D)}_{b,c}$ and $Q^{(N)}_{b,c}$ are Dirichlet forms. Indeed, all regular Dirichlet forms over a discrete measure space are of the form $Q^{(D)}_{b,c}$ for some graph $(X,b,c)$ (cf. \cite{KL12}, Lemma 2.2). This is one motivation to study also graphs that are not locally finite.

As a next step to establish criteria for $Q^{(N)}_\Phi=Q^{(D)}_\Phi$ we will show that the form with magnetic field is dominated by the non-magnetic form. First we prove an easy technical lemma.

\begin{lemma}\label{inequality_sgn}
Let $V$ be a Hilbert space, $a,b\in V$, and $\alpha,\beta\geq 0$ with $\alpha\leq\norm{a},\,\beta\leq\norm{b}$. Define 
\begin{align*}
\tilde a=\begin{cases}\frac{\alpha}{\norm a}a&\colon a\neq 0\\0&\colon a=0\end{cases}
\end{align*}
and likewise $\tilde b$.\\
Then
\begin{align*}
\norm{\tilde a-\tilde b}^2\leq\abs{\alpha-\beta}^2+\norm{a-b}^2.
\end{align*}
\end{lemma}
\begin{proof}
If $a=0$ or $b=0$, the inequality is obvious. Hence assume that $a,b\neq 0$.\\
In the following computation we use the inequality $2\lambda\mu\leq\lambda^2+\mu^2$ for $\lambda,\mu\in\IR$.
\begin{align*}
\norm{\tilde a-\tilde b}^2&=\norm{\tilde a}^2+\norm{\tilde b}^2-2\Re\langle\tilde a,\tilde b\rangle\\
&=\alpha^2+\beta ^2-2\Re\langle \tilde a,\tilde b\rangle\\
&=\abs{\alpha-\beta}^2+2\alpha\beta-\Re \langle\tilde a,\tilde b\rangle\\
&=\abs{\alpha-\beta}^2+2\frac{\alpha\beta}{\norm a\norm b} (\norm a\norm b-\Re\langle a,b\rangle)\\
&\leq \abs{\alpha-\beta}^2+ 2\norm{a}\norm{b}-2\Re\langle a,b\rangle\\
&\leq \abs{\alpha-\beta}^2+\norm{a}^2+\norm{b}^2-2\Re\langle a,b\rangle\\
&=\abs{\alpha-\beta}^2+\norm{a-b}^2\qedhere
\end{align*}
\end{proof}

We will now prove that the magnetic form is dominated by the form without magnetic field. In the form of a pointwise Kato's inequality this result was given in \cite{MT15}, Lemma 3.3.

\begin{proposition}\label{domination_magnetic_forms}
Assume that $\langle W(x)u(x),u(x)\rangle_x\geq c(x)\abs{u(x)}^2$ for all $x\in X,\,u(x)\in F_x$. Then $Q^{(N)}_{\Phi,b,W}$ is dominated by $Q^{(N)}_{b,c}$.
\end{proposition}
\begin{proof}
By the characterization of domination, Corollary \ref{cor_domination},
it suffices to show that $D(Q^{(N)}_{\Phi,b,W})$ is a generalized ideal in $D(Q^{(N)}_{b,c})$ and that
\begin{align*}
\Re Q^{(N)}_{\Phi,b,W}(u,\tilde u)\geq Q^{(N)}_{b,c}(\abs u,\abs {\tilde u})
\end{align*}
holds for all $u,\tilde u\in D(Q^{(N)}_{\Phi,b,W})$ such that $\langle u(x),\tilde u(x)\rangle_x=\abs{u(x)}\abs{\tilde u(x)}$ for all $x\in X$.\\
First, let $u\in D(Q^{(N)}_{\Phi,b,W})$. Then $\abs u\in\ell^2(X,m)$ and
\begin{align*}
\tilde Q_{\Phi,b,W}(u)&=\frac 1 2\sum_{x,y}b(x,y)\abs{u(x)-\Phi_{x,y}u(y)}^2+\sum_x\langle W(x)u(x),u(x)\rangle\\
&\geq \frac 1 2\sum_{x,y}b(x,y)\abs{\abs{u(x)}-\abs{u(y)}}^2+\sum_x c(x)\abs{u(x)}^2\\
&=\tilde Q_{b,c}(\abs u),
\end{align*}
hence $\abs u\in D(Q^{(N)}_{b,c})$.\\
Next let $v\in D(Q^{(N)}_{b,c})$ with $0\leq v\leq\abs u$. Obviously, $\norm{v\sgn u}_{\ell^2}\leq\norm v_{\ell^2}$, thus $v\sgn u\in\ell^2(X,m;F)$.\\
Applying Lemma \ref{inequality_sgn} to $V=F_x,a=u(x),b=\Phi_{x,y}u(y),\alpha=v(x),\beta=v(y)$, we obtain
\begin{align*}
\abs{v(x)\sgn u(x)-\Phi_{x,y}v(y)\sgn u(y)}^2\leq\abs{v(x)-v(y)}^2+\abs{u(x)-\Phi_{x,y}u(y)}^2.
\end{align*}
Summation over $x,y$ implies
\begin{align*}
\tilde Q_{\Phi,b,0}(v\sgn u)\leq  Q^{(N)}_{b,0}(v)+Q^{(N)}_{\Phi,b,0}(u).
\end{align*}
Furthermore,
\begin{align*}
\sum_x\langle W(x)v(x)\sgn u(x),v(x)\sgn u(x)\rangle&\leq\sum_x \abs{u(x)}^2\langle{W(x)\sgn u(x),\sgn u(x)}\rangle\\
&=\sum_x \langle W(x)u(x),u(x)\rangle,
\end{align*}
hence
\begin{align*}
\tilde Q_{\Phi,b,W}(v\sgn u)\leq Q^{(N)}_{b,0}(v)+Q^{(N)}_{\Phi,b,0}(u)+\sum_x c(x)\abs{u(x)}^2\leq Q^{(N)}_{b,c}(v)+Q^{(N)}_{\Phi,b,W}(u),
\end{align*}
that is, $v\sgn u\in D(Q^{(N)}_{\Phi,b,W})$.\\
Let $u,\tilde u\in D(Q^{(N)}_{\Phi,b,W})$ such that $\langle u(x),\tilde u(x)\rangle_x=\abs{u(x)}\abs{\tilde u(x)}$ for all $x\in X$. Then we have
\begin{align*}
&\Re\langle u(x)-\Phi_{x,y}u(y),\tilde u(x)-\Phi_{x,y}\tilde u(x)\rangle\\
={}&\Re(\langle u(x),\tilde u(x)\rangle-\langle u(x),\Phi_{x,y}\tilde u(y)\rangle-\langle\Phi_{x,y} u(y),\tilde u(x)\rangle+\langle u(y),\tilde u(y)\rangle)\\
={}&\abs{u(x)}\abs{\tilde u(x)}+\abs{u(y)}\abs{\tilde u(y)}-\Re\langle u(x),\Phi_{x,y}\tilde u(y)\rangle-\Re\langle\Phi_{x,y} u(y),\tilde u(x)\rangle\\
\geq{}& \abs{u(x)}\abs{\tilde u(x)}+\abs{u(y)}\abs{\tilde u(y)}-\abs{u(x)}\abs{\tilde u(y)}+\abs{u(y)}\abs{\tilde u(x)}\\
={}&(\abs{u(x)}-\abs{u(y)})(\abs{\tilde u(x)}-\abs{\tilde u(y)}):
\end{align*}
After multiplication with $b(x,y)$ and summation over $x,y\in X$ we get
\begin{align*}
\Re Q^{(N)}_{\Phi,b, W}(u,\tilde u)\geq Q^{(N)}_{b,c}(\abs u,\abs{\tilde u}).&\qedhere
\end{align*}
\end{proof}

\begin{corollary}
The form $Q_{b,c}^{(N)}$ 
is dominated by $Q^{(N)}_{b,0}$. 
\end{corollary}
\begin{proof}
This follows from Proposition \ref{domination_magnetic_forms} by taking $F_x=\IC$, $W(x)=c(x)$ and $\Phi_{x,y}=1$ for all $x,y\in X$.
\end{proof}

Having proven the domination property, the announced main result of this section is now an easy consequence of Corollary \ref{corollary_uniqueness_forms}. In a very informal way it says that adding a magnetic and electric field does not disturb the form uniqueness.

\begin{theorem}\label{form_uniqueness_graphs}
Assume that $\langle W(x)u(x),u(x)\rangle\geq c(x)\abs{u(x)}^2$ for all $x\in X,\,u(x)\in F_x$. If $Q^{(D)}_{b,c}=Q^{(N)}_{b,c}$, then $Q^{(D)}_{\Phi,b,W}=Q^{(N)}_{\Phi,b,W}$.
\end{theorem}
\begin{proof}
The form $Q^{(D)}_{b,c}$ is a regular Dirichlet form so that we only have to check assumptions (a) and (b) of Corollary \ref{corollary_uniqueness_forms}.\\
Assumption (a): For $v\in C_c(X)$ define $u:=v e_1\in \Gamma_c(X;F)$. Then $\abs u=\abs v$.\\
Assumption (b): By Proposition \ref{domination_magnetic_forms}, $Q^{(N)}_{\Phi,b,W}$ is dominated by $Q^{(N)}_{b,c}$.
\end{proof}

There are quite a few conditions under which $Q^{(D)}_{b,c}=Q^{(N)}_{b,c}$ holds. The first were phrased in terms of the measure $m$ and the combinatorial graph structure:

\begin{example}
If $\tilde L C_c(X)\subset \ell^2(X,m)$ and $\sum_{n=1}^\infty m(x_n)=\infty$ for any sequence $(x_n)$ in $X$ such that $b(x_n,x_{n+1})>0$ for all $n\in \IN$, then $L_0:=\tilde L|_{C_c(X)}$ is essentially self-adjoint. In particular, if $\inf_{x\in X}m(x)>0$, both assumptions hold (cf. \cite{KL12}, Theorem 6).\\
Consequently, all self-adjoint extensions of $L_0$ coincide and therefore also all form extensions of $Q^{(D)}$.
\end{example}

It turned out that the concept of intrinsic pseudo metrics (cf. \cite{FLW} for a discussion of intrinsic metrics for non-local Dirichlet forms in general) provides a suitable framework for many conditions for uniqueness of form extensions.

A pseudo metric $d\colon X\times X\lra [0,\infty)$ is called intrinsic if
\begin{align*}
\frac 1{m(x)}\sum_ yb(x,y) d(x,y)^2\leq 1
\end{align*}
for all $x\in X$.

A pseudo metric $d$ is said to have finite jump size if there is an $s\in \IR$ such that $b(x,y)=0$ for all $x,y\in X$ with $d(x,y)>s$.

A pseudo metric is called a path pseudo metric if there is a function $\sigma\colon X\times X\lra[0,\infty)$, satisfying $\sigma(x,y)=\sigma(y,x)$ and $\sigma(x,y)>0$ iff $b(x,y)>0$ for all $x,y\in X$, such that
\begin{align*}
d(x,y)=d_\sigma(x,y):=\inf_{\gamma}\sum_{k=1}^n \sigma(x_{k-1},x_k)
\end{align*}
where the infimum is taken over all paths $(x_0,\ldots,x_n)$ connecting $x$ and $y$.

An intrinsic path pseudo metric $d_\sigma$ is called strongly intrinsic if
\begin{align*}
\frac 1{m(x)}\sum_y b(x,y)\sigma(x,y)^2\leq 1
\end{align*}
for all $x\in X$.

The following conditions are taken from \cite{HKMW}, Theorem 1 and 2. Further examples can be found there.

\begin{example}
Let $d$ be an intrinsic pseudo metric on $(X,b,0,m)$. If the weighted degree function 
\begin{align*}
\Deg\colon X\lra [0,\infty),\,\Deg(x)=\frac 1 {m(x)}\left(\sum_{y\in X}b(x,y)+c(x)\right)
\end{align*}
is bounded on the combinatorial neighborhood of each distance ball, then $Q^{(D)}=Q^{(N)}$.
\end{example}

\begin{example}
If $(X,b,0,m)$ is locally finite and there is an intrinsic path metric $d$ such that $(X,d)$ is metrically complete, then $L_0$ is essentially self-adjoint and consequently $Q^{(D)}=Q^{(N)}$.
\end{example}

The following condition is given in \cite{HL15}, Lemma 2.5, and \cite{AT15}, Theorem 1. Its connection to intrinsic metrics is discussed in \cite{HL15}, Theorem 2.7.
\begin{example}
If the graph is \emph{complete}, then $Q^{(D)}=Q^{(N)}$. Here, completeness means that there is a non-decreasing sequence $(\eta_k)$ in $C_c(X)$ such that $\eta_k\to 1$ pointwise and
\begin{align*}
\frac 1{m(x)}\sum_y b(x,y)\abs{\eta_k(x)-\eta_k(y)}^2\leq \frac 1 k
\end{align*}
for all $x\in X,\,k\in\IN$.\\
That $Q^{(D)}_{\Phi,b,W}=Q^{(N)}_{\Phi,b,W}$ holds under this condition seems to be new.
\end{example}

\subsection{Magnetic Schrödinger forms on domains in Euclidean space}\label{Magnetic Schrödinger operators in Euclidean space}
In this section we will show that coincidence of the minimal and maximal form for the Laplacian defined below implies the coincidence of minimal and maximal form of a magnetic Schrö\-dinger operator on domains in Eu\-clid\-e\-an space. These operators and questions of uniqueness of extension have been extensively studied; for background information we refer the reader to \cite{Kat72, Sim79, HS04}. Of course, this list is not comprehensive and probably not even representative at all.

Contrary to the discrete case treated in the last section, some more work has to be done in order to derive the uniqueness result, Theorem \ref{uniqueness_magnetic_forms}, from Corollary \ref{uniqueness_smooth}. This is due to the fact that for a form on a domain $\Omega\subset \IR^n$, smooth inner regularity is a non-void condition.

In the first part of this section we will define the basic objects of our interest, magnetic Schrödinger forms (Definitions \ref{def_magnetic_form_Neumann} and \ref{def_magnetic_form_Dirichlet}), and establish that they are closed. In the second part we show that the Schrödinger form with magnetic field is dominated by the Schrödinger form without magnetic field (Proposition \ref{domination_domains}). In the last part we present the announced uniqueness result (Theorem \ref{uniqueness_magnetic_forms}) for magnetic forms on Euclidean domains and give some examples.

Mirroring these technical difficulties, there is some more notation to fix than in the last sections: 

In the following, $\Omega\subset \IR^n$ shall denote a domain. For a subset $K\subset\Omega$ we write $K\subset\subset \Omega$ if $\overline K$ is compact and contained in $\Omega$.\\
If $f$ is weakly differentiable, then $\nabla f=(\partial_1 f,\ldots,\partial_n f)$ denotes the vector valued function with the weak derivatives of $f$ as components.

We will use the following notations for function spaces: 
\begin{itemize}
\item $C_c^\infty(\Omega)$, the space of smooth, compactly supported functions on $\Omega$,
\item $\D^\prime(\Omega)$, the space of distributions on $\Omega$,
\item $L^p_{\loc}(\Omega)=\{u\colon\Omega\lra\IC\mid u\text{ measurable}, u|_K\in L^p(K)\text{ for all }K\subset\subset \Omega\}$,
\item $L^p_{c}(\Omega)=\{u\in L^p(\Omega)\mid u|_{K^c}=0\text{ for some }K\subset\subset \Omega\}$,
\item $W^{1,p}(\Omega)=\{u\in L^p(\Omega)\mid\nabla u\in L^p(\Omega;\IC^n)\}$,
\item $W^{1,p}_{\loc}(\Omega)=\{u\in L^p_{\loc}(\Omega)\mid \nabla u\in L^p_{\loc}(\Omega;\IC^n)\}$.
\end{itemize}

In addition, the corresponding vector- and matrix-valued spaces will be denoted by $L^p_{\loc}(\Omega;\IC^n), L^p_{\loc}(\Omega;M_n(\IR))$ etc. By $M_n(\IR)^+$ we denote the set of real-valued, symmetric, positive matrices.

\begin{definition}[Magnetic form with Neumann boundary conditions]\label{def_magnetic_form_Neumann}
Let $a\in L^\infty_{\loc}(\Omega;M_n(\IR)^+)$, $b\in L^2_\loc(\Omega;\IR^n)$, $V\in L^1_\loc(\Omega)^+$. Let $D_k\colon L^2(\Omega)\lra \D^\prime(\Omega),\,D_k u=\partial_k u-i b_k u$, $Du=(D_1 u,\ldots,D_n u)$. Assume that for every $K\subset\subset \Omega$ there is a constant $\mu_K>0$ such that $\mu_K 1_n\leq a(x)$ for almost all $x\in K$.\\
Define the magnetic form with Neumann boundary conditions via
\begin{align*}
D(\E_{a,b,V}^{(N)})&=\{u\in L^2(\Omega)\mid a^{\frac 1 2}Du\in L^2(\Omega;\IC^n), V^{\frac 1 2}u\in L^2(\Omega)\},\\
\E_{a,b,V}^{(N)}(u,v)&= \sum_{j,k=1}^n\int_\Omega a_{jk}(D_k u)\overline{(D_k v)}\,dx+\int_\Omega Vu\overline v\, dx.
\end{align*}
The form norm of $\E^{(N)}_{a,b,V}$ will be denoted by $\norm\cdot_{a,b,V}$.
\end{definition}

For the remainder of the section we shall always assume the regularity assumptions on $a$ and $b$ without further mentioning it. The regularity assumption on $V$ will be weakened later in order to allow negative potentials $V$. We will denote by $1$ the constant function $\Omega\to M_n(\IR), x\mapsto 1_n$ (the identity matrix).

\begin{remark}
\begin{itemize}
\item If the coefficients are sufficiently regular, the form $\E_{a,b,V}$ can be viewed as form associated with the differential expression
\begin{align*}
\tau=-\sum_{j,k}(\partial_k-ib_k)a_{jk}(\partial_i-ib_i)+V.
\end{align*}
However, the approach via forms allows us to handle the case of coefficients that are not (weakly) differentiable and for which the expression $\tau$ makes no immediate sense.
\item The expression $\tau$ is elliptic with principal part
\begin{align*}
-\sum_{j,k}a_{jk}\partial_j\partial_k.
\end{align*}
The condition $a\geq \mu_K$ on compact subsets $K\subset \Omega$ then translates to the condition of $\tau$ being locally strongly elliptic.
\item In quantum mechanics, the expression $\E_{1,b,V}(\psi)$ is the energy of a particle with wave function $\psi$ in an electric field with potential $V$ and a magnetic field with magnetic potential $b$.
\end{itemize}
\end{remark}
We will prove that the forms $\E^{(N)}_{a,b,V}$ are closed in several steps by ``turning on'' the fields $b,a,V$ successively.

\begin{lemma}\label{magnetic_form_closed}
The form $\E^{(N)}_{1,b,0}$ is closed.
\end{lemma}
\begin{proof}
Let $(u_m)$ be  a Cauchy sequence with respect to the form norm and $u_m\to u$ in $L^2(\Omega)$. Then $(D  u_m)_m$ is a Cauchy sequence in $L^2(\Omega;\IC^n)$. Hence, there is a function $v\in L^2(\Omega;\IC^n)$, such that $D u_m\to v$ in $L^2(\Omega;\IC^n)$. Define $w:=v+iu b$. Then we have for all $\phi\in C_c^\infty(\Omega)$, $j\in\{1,\dots,n\}$:
\begin{align*}
\int_\Omega w_j \phi\,dx&=\int_\Omega v_j\phi\,dx+i\int_\Omega b_j u\phi\,dx\\
&=\lim_{m\to\infty}\left(\int_\Omega D_j u_m \phi\,dx+\int_\Omega i b_j u_m \phi\,dx\right)\\
&=\lim_{m\to\infty}\int_\Omega \partial_j u_m \phi\,dx\\
&=-\lim_{n\to\infty}\int_\Omega u_m\partial_j\phi\,dx\\
&=-\int_\Omega u\partial_j\phi\,dx.
\end{align*}
Thus, $w_j=\partial_j u$ and $\partial_j u-ib_j u=v_j\in L^2$, hence $u\in D(\E^{(N)}_{b,0})$. By definition, $D_j u_m\to v_j=\partial_j u-ib_j u$, that is, $\norm{u-u_m}_{1,b,0}\to 0$.
\end{proof}

\begin{lemma}
The form $\E^{(N)}_{a,b,0}$ is closed.
\end{lemma}
\begin{proof}
Let $(u_m)$ be a Cauchy sequence with respect to $\norm\cdot_{a,b,0}$ and $u_m\to u$ in $L^2(\Omega)$. Let $K\subset\subset \Omega$ be open. Then we have
\begin{align*}
\E^{(N)}_{a,b,0}(u_m-u_l)\geq\mu_K\int_K \abs{D(u_m-u_l)}^2\,dx.
\end{align*}
By Lemma \ref{magnetic_form_closed}, there is a $u^{(K)}\in L^2(K)$ such that $Du^{(K)}\in L^2(K)$ and $\norm{u_m|_{K}-u^{(K)}}_{L^2(K)}^2+\norm{D(u_m|_{K}-u^{(K)})}_{L^2(K)}^2\to 0$. Thus, $u^{(K)}=u|_K$.\\
Let $K_1\subset K_2\subset\subset \Omega$ be open subsets and $\phi\in C_c^\infty(\Omega)$ with $\supp \phi\subset K_1$. Then we have
\begin{align*}
\int_{K_1} (\nabla u^{(K_2)}-\nabla u^{(K_1)})\phi\,dx&=\int_{K_2}\nabla u^{(K_2)}\phi\,dx-\int_{K_1}\nabla u^{(K_1)}\phi\,dx\\
&=-\int_{K_2}u\nabla \phi\,dx+\int_{K_1}u\nabla \phi\,dx\\
&=-\int_{K_1}u\nabla \phi\,dx+\int_{K_1}u\nabla \phi\,dx\\
&=0.
\end{align*}
Hence $\nabla u^{(K_2)}$ and $\nabla u^{(K_1)}$ coincide on $K_1$ and we can define $v\in L^1_{\loc}(\Omega;\IC^n)$ via $v|_K=\nabla u^{(K)}$.\\
Now let $\phi\in C_c^\infty(\Omega)$. Then there is an open subset $K\subset\subset \Omega$ with $\supp\phi\subset K$ and we have
\begin{align*}
\int_\Omega \nabla u\phi\,dx=\int_K\nabla u^{(K)}\phi\,dx=-\int_K u\nabla\phi\,dx=-\int_{\Omega}u\nabla\phi\,dx.
\end{align*}
Thus, $u$ is weakly differentiable and $\nabla u=v$.\\
As $Du_m|_k\to Du|_K$ in $L^2(K)$ for all $K\subset\subset\Omega$, there is a subsequence that converges almost everywhere in $K$, and as $\IR^n$ is $\sigma$-compact, we may assume w.\,l.\,o.\,g. that $Du_m\to Du$ almost everywhere.
By Fatou's Lemma we obtain
\begin{align*}
\int_\Omega\sum_{j,k}a_{jk}(D_j u)(\overline{D_k u})\,dx&=\int_\Omega\lim_{m\to\infty}\sum_{j,k}a_{jk}(D_j u_m)(\overline{D_k u_m})\,dx\\
&\leq \liminf_{m\to\infty}\E^{(N)}_{a,b,0}(u_m)\\
&<\infty.
\end{align*}
Thus, $u\in D(\E^{(N)}_{a,b,0})$.\\
By the same argument we have
\begin{align*}
\E^{(N)}_{a,b,0}(u-u_m)&=\int_\Omega\sum_{j,k} a_{jk}(D_j(u-u_m))(\overline{D_k(u-u_m)})\,dx\\
&=\int_\Omega \lim_{l\to\infty}\sum_{j,k}a_{jk}(D_j(u_l-u_m))(\overline{D_k(u_l-u_m)})\,dx\\
&\leq\liminf_{l\to \infty}\E^{(N)}_{a,b,0}(u_l-u_m)\\
&\to 0,\,m\to\infty.
\end{align*}
Hence, $\norm{u-u_m}_{a,b,0}\to 0$.
\end{proof}

\begin{lemma}
Let $V\in L^1_\loc(\Omega)^+$. Then the form $\E^{(N)}_{a,b,V}$ is closed.
\end{lemma}
\begin{proof}
The form $\E^{(N)}_{a,b,V}$ is the form sum of the closed forms $\E^{(N)}_{a,b,0}$ and $q_V$ given by
\begin{align*}
q_V(u)=\int_\Omega V\abs{u}^2\,dx.
\end{align*}
Hence, it is also closed.
\end{proof}

\begin{definition}[Form small potential]
Let $q$ be a closed form on $L^2(\Omega)$. A function $V\colon \Omega\lra[0,\infty)$ is called form small with respect to $q$ if there are constants $\alpha\in[0,1),\,\beta\in[0,\infty)$ such that
\begin{align*}
\int_\Omega V\abs{u}^2\,dx\leq \alpha q(u)+\beta\norm{u}_{L^2}^2
\end{align*}
for all $u\in L^2(\Omega)$. 
\end{definition}

\begin{definition}[Magnetic form with Dirichlet and Neumann boundary conditions]\label{def_magnetic_form_Dirichlet}
Let $V\colon \Omega\lra \IR$ be measurable such that $V^+\in L^1_\loc(\Omega)$ and $V^-$ is form small with respect to $\E^{(N)}_{a,0,0}$. Then the magnetic form with Neumann conditions is defined via
\begin{align*}
D(\E^{(N)}_{a,b,V})=D(\E^{(N)}_{a,b,V^+}),\,\E^{(N)}_{a,b,V}(u)=\sum_{j,k=1}^n\int_\Omega a_{jk}(D_j u)\overline{(D_k u)}\,dx+\int_\Omega V\abs{u}^2\, dx
\end{align*}
and the magnetic form with Dirichlet boundary conditions $\E^{(D)}_{a,b,V}$ as the closure of the restriction of $\E^{(N)}_{a,b,V}$ to $C_c^\infty(\Omega)$.
\end{definition}

From now on we shall always assume that $V^+\in L^1_{\loc}(\Omega)$ and that $V^-$ is form small with respect to $\E^{(N)}_{a,0,0}$ for the remainder of the section.\\
This condition on $V$ ensures that the form $\E^{(N)}_{a,b,V}$ is closed and the form $\E^{(D)}_{a,b,V}$ therefore closable by the KLMN Theorem (cf. \cite{RS2}, Thm. X.16). For a more detailed discussion we refer once again to \cite{HS04}.

For positive potential $V$, the forms without magnetic field $\E^{(N)}_{a,0,V}$ and $\E^{(D)}_{a,0,V}$ are Dirichlet forms (cf. \cite{Fu} , Examples 1.2.1 and 1.2.3) and the form with Dirichlet boundary conditions is regular by definition. This fact is one of the connections between the Euclidean case and the discrete case treated in Section \ref{Magnetic Schrödinger operators on graphs}.

Now that we have defined the Schrödinger forms, we will in a next step show that the form with magnetic field is dominated by the form without magnetic field (indeed, some variation in the potential is also allowed). To do so, we start with two technical lemmas, the first one giving some regularity results and the second one providing a product rule for weak derivatives.

\begin{lemma}\label{regularity_magnetic}
\begin{itemize}
\item[(a)]If $u\in D(\E^{(N)}_{a,0,V})$, then $\nabla u\in L^2_{\loc}(\Omega;\IC^n)$.
\item[(b)]If $u\in D(\E^{(N)}_{a,b,V})\cap L^\infty(\Omega)$, then $\nabla u\in L^2_{\loc}(\Omega;\IC^n)$.
\end{itemize}
\end{lemma}
\begin{proof}
\begin{itemize}
\item[(a)]By definition, $a^\frac 1 2\nabla u\in L^2(\Omega;\IC^n)$. For every $K\subset\subset \Omega$ there is a constant $\mu_K>0$ such that $a\geq \mu_K$ on $K$, so $a^{-\frac 1 2}\in L^\infty_{\loc}(\Omega;M_n(\IR)^+)$. It follows that $\nabla u=a^{-\frac 1 2} a^{\frac 1 2} \nabla u\in L^2_{\loc}(\Omega;\IC^n)$.
\item[(b)]Let $u\in D(\E^{(N)}_{a,b,V})\cap L^\infty(\Omega)$. Since $u\in L^\infty(\Omega)$, $b\in L^2_{\loc}(\Omega;\IR^n)$ and $a\in L^\infty_{\loc}(\Omega;M_n(\IR)^+)$, we have  $u a^{\frac 1 2} b\in L^2_{\loc}(\Omega;\IC^n)$.\\
By definition, $a^{\frac 1 2}Du\in L^2(\Omega;\IC^n)$.\\
Thus, we can conclude $a^{\frac 1 2}\nabla u=a^{\frac 1 2}Du+iua^{\frac 1 2}b\in L^2_{\loc}(\Omega;\IC^n)$. The same argument as in (a) gives $\nabla u\in L^2_{\loc}(\Omega;\IC^n)$.\qedhere
\end{itemize}
\end{proof}

The following Leibniz rule for weak derivatives is taken from \cite{Man01}, Hilfssätze 14.1.1. and 14.1.2. As the proof s are quite technical, we will not reproduce them here.
\begin{lemma}\label{prod_weak_derivatives}
\begin{itemize}
\item[(a)]Product rule for weak derivatives: If $u,v\in W^{1,1}_{\loc}(\Omega)$ such that $u\nabla v$, $v\nabla u\in L^1_{\loc}(\Omega;\IC^n)$, then $uv\in W^{1,1}_{\loc}(\Omega)$ and
\begin{align*}
\nabla (uv)=u\nabla v+v\nabla u.
\end{align*}
\item[(b)]If $u\in W^{1,1}_{\loc}(\Omega)$, then $\abs{u}\in W^{1,1}_{\loc}(\Omega)$ and
\begin{align*}
\abs{u}\nabla\abs{u}=\Re(\bar u\nabla u).
\end{align*}
\end{itemize}
\end{lemma}

The next proposition shows that the form with magnetic potential is dominated by the form without magnetic potential. In principle, this fact has been long known and probably goes back to Simon (cf. \cite{Sim79}), just the regularity assumptions on $a$, $b$ and $V$ have been considerably weakened over the last decades. In the form presented here, the statement is taken from \cite{HS04}, Theorem 3.3. Our proof is a little different in that we prove the domination property for the forms instead of the semigroups.
\begin{proposition}\label{domination_domains}
If $\tilde V\colon\Omega\to\IR$ is measurable, $\tilde V\leq V$ and $\tilde V^-$ is form small with respect to $\E^{(N)}_{a,0,0}$, then $\E_{a,b,V}^{(N)}$  is dominated by $\E_{a,0,\tilde V}^{(N)}$.
\end{proposition}
\begin{proof}
In the first step we show that $D(\E^{(N)}_{a,b,V})$ is a generalized ideal of $D(\E^{(N)}_{a,0,\tilde V})$.\\
Let $u\in D(\E^{(N)}_{a,b,V})\subset W^{1,1}_{\loc}(\Omega)$. By Lemma \ref{prod_weak_derivatives} we have $\abs{u}\in W^{1,1}_{\loc}(\Omega)$ and 
\begin{align*}
\abs{u}\nabla\abs{u}&=\Re(\bar u\nabla u)=\Re(\bar u\nabla u-i\abs{u}^2b)=\Re(\bar u Du).\tag{$\dagger$}\label{real_part_magnetic}
\end{align*}
Let $\xi\in \IC^n$ and $x\in \Omega$. Then
\begin{align*}
\sum_{j,k} a_{j,k}(x)\Re \xi_j \Re \xi_k&\leq \sum_{j,k}a_{jk}(x)(\Re\xi_j)(\Re\xi_k)+\sum_{j,k} a_{jk}(x)(\Im \xi_j)(\Im \xi_k)\\
&=\Re\sum_{j,k} a_{jk}(x)\xi_j\bar \xi_k.
\end{align*}
Applied to $\xi=\overline{u(x)} Du(x)$, we obtain
\begin{align*}
\sum_{j,k} a_{j,k}(\partial_j \abs{u})(\partial_k \abs{u})\leq\Re\sum_{j,k}a_{j,k}(D_j u)(\overline{D_k u}).\tag{$\dagger\dagger$}\label{ineq_forms}
\end{align*}
This implies $\abs{u}\in D(\E^{(N}_{a,0,V})$.\\
Next let $f_1,f_2\in L^2(\Omega)$ such that $\nabla\abs{f_1},\nabla\abs{f_2}\in L^2_{\loc}(\Omega;\IC^n)$. Assume that $\bar f_1 f_2=\abs{f_1}\abs{f_2}$ and notice that this condition is equivalent to $\abs{f_1} f_2=f_1\abs{f_2}$. Differentiating this identity we obtain
\begin{align*}
\abs{f_1}\nabla f_2+f_2\nabla\abs{f_1}=f_1\nabla\abs{f_2}+\abs{f_2}\nabla f_1.
\end{align*}
If $f_1\in D(\E^{(N)}_{a,b,V}),\,v\in D(\E^{(N)}_{a,0,V})^+$ such that $v\leq \abs{f_1}$ and $f_2=v\sgn f_1$, then $\abs{f_1},\abs{f_2}\in D(\E^{(N)}_{a,0,V})$ and by Lemma \ref{regularity_magnetic} the condition $\nabla\abs{f_1},\nabla\abs{f_2}\in L^2_{\loc}(\Omega;\IC^n)$ is met. Thus,
\begin{align*}
\abs{f_1} D f_2=f_1\nabla\abs{f_2}+\abs{f_2}D f_1-f_2\nabla\abs{f_1}\tag{$\dagger\dagger\dagger$}\label{eq_forms}
\end{align*}
and consequently
\begin{align*}
\abs{a^{\frac 1 2}D f_2}\leq\abs{a^{\frac 1 2}\nabla v}+\abs{a^{\frac 1 2}Df_1}+\abs{a^{\frac 1 2}\nabla v}.
\end{align*}
Since all the summands on the right-hand side are in $L^2$ by assumption, so is $a^{\frac 1 2}\ Df_2$. Furthermore,
\begin{align*}
\int_\Omega V\abs{f_2}^2\,dx\leq \int_\Omega V\abs{f_1}^2\,dx<\infty,
\end{align*}
hence $f_2=v\sgn u\in D(\E^{(N)}_{a,b,V})$.\\
Now let $f_1,f_2\in D(\E^{(N)}_{a,b,V})$ such that $\abs{f_1}f_2=f_1\abs{f_2}$. By (\ref{eq_forms}) we have
\begin{align*}
\abs{f_1}^2\Re\sum_{j,k} a_{jk} (D_j f_2)(\overline{D_k f_1})
=\Re\sum_{j,k}a_{jk} \abs{f_1} (\overline{D_k f_1})(f_1\partial_j \abs{f_2}+\abs{f_2} D_j f_1 -f_2 \partial_j \abs{f_1}).
\end{align*}
Two applications of (\ref{real_part_magnetic}) yield
\begin{align*}
&\Re\sum_{j,k}a_{jk} \abs{f_1} (\overline{D_k f_1})(f_1\partial_j \abs{f_2}+\abs{f_2} D_j f_1 -f_2 \partial_j \abs{f_1})\\
={}&\sum_{j,k}a_{jk}\abs{f_1}^2(\partial_j \abs{f_2})(\partial_k \abs{f_1})
+\Re\abs{f_1}\abs{f_2}\sum_{j,k}a_{j,k}(D_j f_1)(\overline{D_k f_1})\\
&-\Re \sum_{j,k} a_{jk}\abs{f_2} f_1(\overline{D_k f_1})\partial_j \abs{f_1}\\
={}&\abs{f_1}^2\sum_{j,k} a_{jk}(\partial_j \abs{f_2})(\partial_k\abs{f_1})+\abs{f_1}\abs{f_2}\Re\sum_{j,k}a_{jk}(D_j f_1)(\overline{D_k f_1})\\
&-\abs{f_1}\abs{f_2}\sum_{j,k} a_{jk}(\partial_j\abs{f_1})(\partial_k \abs{f_1}).
\end{align*}
Using (\ref{ineq_forms}), we conclude
\begin{align*}
&\abs{f_1}^2\Re\sum_{j,k} a_{jk} (D_j f_2)(\overline{D_k f_1})\\
={}&\abs{f_1}^2\sum_{j,k} a_{jk}(\partial_j \abs{f_2})(\partial_k\abs{f_1})+\abs{f_1}\abs{f_2}\Re\sum_{j,k}a_{jk}(D_j f_1)(\overline{D_k f_1})\\
&-\abs{f_1}\abs{f_2}\sum_{j,k} a_{jk}(\partial_j\abs{f_1})(\partial_k \abs{f_1})\\
\geq{}& \abs{f_1}^2\sum_{j,k}a_{jk}(\partial_j \abs{f_2})(\partial_k \abs{f_1}).
\end{align*}
Therefore,
\begin{align*}
\Re\E^{(N)}_{a,b,V}(f_1,f_2)&=\Re\E^{(N)}_{a,b,0}(f_1,f_2)+\int_\Omega Vf_1 \overline{f_2}\,dx\\
&\geq \E^{(N)}_{a,0,0}(S(f_1),S(f_2))+\int_\Omega \tilde V \abs{f_1}\abs{f_2}\,dx\\
&=\E^{(N)}_{a,0,\tilde V}(\abs{f_1},\abs{f_2}).\qedhere
\end{align*}
\end{proof}

\begin{remark}
Notice that while domination is characterized by an inequality for integrals, our proof gives a \emph{pointwise} inequality generalizing the diamagnetic inequality (cf. \cite{LL97}, Theorem 7.21)
\begin{align*}
\abs{Du}\geq \abs{\nabla\abs{u}}.
\end{align*}
In particular, the regularity requirements are only needed to assure that Lemma \ref{prod_weak_derivatives} is applicable, the rest of the proof is purely algebraic in means.
\end{remark}

Having proven the domination, we will now show that $\E^{(N)}_{a,b,V}$ is smoothly inner regular so that we can apply Corollary \ref{uniqueness_smooth}.

\begin{proposition}\label{smooth_reg}
The form $\E^{(N)}_{a,b,V}$ is smoothly inner regular.
\end{proposition}
\begin{proof}
Let $\eta\in C_c^\infty(\IR^n)$ such that $\supp \eta\subset \overline{B_1(0)}$, $\eta\geq 0$ and $\norm\eta_{L^1}=1$, and define $\eta_\epsilon(x)=\epsilon^{-n}\eta(\epsilon^{-1}x)$.\\
Let $u\in D(\E^{(N)}_{a,b,V})\cap L^\infty_c(\Omega)$ and $K\subset\subset \Omega$ be open such that $u|_{K^c}=0$.  Let 
\begin{align*}
K_\epsilon=\{x\in\Omega\mid d(x, K)< \epsilon\}
\end{align*}
for $\epsilon<d( K,\Omega^c)$ and $\delta=\frac 1 2 d(K,\Omega^c)$.\\
By Lemma \ref{regularity_magnetic}, $a^{\frac 1 2}\nabla u\in L^2_{\loc}(\Omega;\IC^n)$. Since $u|_{K^c}=0$, we have that  $\norm{a ^{\frac  12}\nabla u}_{L^2(\Omega)}= \norm{a^{\frac  12} \nabla u}_{L^2(K_\delta)}<\infty$. Moreover, $a\geq \mu_{K_\delta}$ on $K_\delta$ implies $a^{-\frac 1 2}\leq \mu_{K_\delta}^{-\frac 1 2}$ on $K_\delta$, hence $\nabla u=a^{-\frac 12}a^{\frac 1 2}\nabla u\in L^2(K_\delta)$. Since $\nabla u$ vanishes outside $K_\delta$, we have $u\in H^1(\Omega)$.
Extend $u$ by $0$ to $\IR^n$ and define $u_\epsilon=(u\ast \eta_\epsilon)|_{\Omega}$ for $\epsilon<\delta$. Then $u_\epsilon\in C_c^\infty(\Omega)$, $\supp u_\epsilon\subset K_\epsilon\subset K_\delta$ and $u_\epsilon\to u,\epsilon\downarrow 0,$ in $H^1(\Omega)$ and pointwise by the mollifier theorem (cf. \cite{LL97}). Moreover, $\abs{u_\epsilon(x)}\leq\norm u_\infty$. for all $x\in\Omega$.\\
Since $a\in L^\infty(K_\delta;M_n(\IR)^+)$, we have
\begin{align*}
\norm{a^{\frac 1 2}(\nabla u-\nabla u_{\epsilon})}_{L^2(K_\delta)}\leq \norm{a^{\frac 1 2}}_{L^\infty(K_\delta)}\norm{u-u_\epsilon}_{H_1(K_\delta)}\to 0,\,\epsilon\downto 0.
\end{align*}
Since $a\in L^\infty(K_\delta;M_n(\IR)^+),b\in L^2(K_\delta;\IR^n)$ and $V^+\in L^1(K_\delta)$ and $\abs{u_\epsilon(x)}\leq\norm{u}_\infty$ for almost all $x\in\Omega$, an application of the dominated convergence theorem yields
\begin{align*}
\norm{(u-u_\epsilon)a^{\frac 1 2}b}_{L^2}^2&=\int_{K_\delta}\abs{(u_\epsilon-u)a^{\frac 1  2}b}^2\,dx\to 0,\,\epsilon\downto 0,\\
\norm{(V^+)^{\frac 1 2}\abs{u-u_\epsilon}}_{L^2}^2&=\int_{K_\delta}V^+\abs{u-u_\epsilon}^2\,dx\to 0,\,\epsilon\downto 0.
\end{align*}
Therefore, we have
\begin{align*}
\E^{(N)}_{a,b,V^+}(u-u_\epsilon)&=\norm{a^{\frac 1 2}D(u-u_\epsilon)}^2_{L^2}+\norm{(V^+)^{\frac 1 2}(u-u_\epsilon)}^2_{L^2}\\
&\leq 2\norm{a^{\frac 1 2}\nabla (u-u_\epsilon)}_{L^2}^2+\norm{(u-u_\epsilon)a^{\frac 1 2}b}_{L^2}^2+\norm{(V^+)^{\frac 1 2}\abs{u-u_\epsilon}^2}_{L^2}^2\\
&\to 0,\,\epsilon\downarrow 0.
\end{align*}
Finally, since $V^-$ is form small with respect to $\E^{(N)}_{a,b,V^+}$, we also have $\E^{(N)}_{a,b,V}(u-u_\epsilon)\to 0,\,\epsilon\downto 0$.\\
Thus, $D(\E^{(N)}_{a,b,V})\cap L^\infty_c(\Omega)\subset \overline{C_c^\infty(\Omega)}^{\norm\cdot_{a,b,V}}=D(\E_{a,b,V}^{(D)})$.
\end{proof}

The main theorem of this section is now an immediate consequence of the abstract results in Chapter \ref{chap_uniqueness_criterion} and the preceding propositions.

\begin{theorem}\label{uniqueness_magnetic_forms}
Let $\tilde V\colon \Omega\lra\IR^n$ be measurable such that $\tilde V\leq V$ and $\tilde V^-$ is relatively form bounded with respect to $\E^{(N)}_{a,0,0}$. If $\E^{(D)}_{a,0,\tilde V}=\E^{(N)}_{a,0,\tilde V}$, then $\E^{(D)}_{a,b,V}=\E^{(N)}_{a,b,V}$.
\end{theorem}
\begin{proof}
By Proposition \ref{domination_domains}, $\E^{(N)}_{a,b,V}$ is dominated by $\E^{(N)}_{a,0,\tilde V}$. By Proposition \ref{smooth_reg}, the form $\E^{(N)}_{a,b,V}$ is smoothly inner regular. Thus, an application of Corollary \ref{uniqueness_smooth} yields $\E^{(N)}_{a,b,V}=\E^{(D)}_{a,b,V}$.
\end{proof}

\begin{lemma}\label{negative_part_negligible}
If $\E^{(D)}_{a,b,V^+}=\E^{(N)}_{a,b,V^+}$, then $\E^{(D)}_{a,b,V}=\E^{(N)}_{a,b,V}$.
\end{lemma}
\begin{proof}
Let $u\in D(\E^{(N)}_{a,b,V})=D(\E^{(N)}_{a,b,V^+})=D(\E^{(D)}_{a,b,V^+})$. By assumption there is a sequence $\phi_n\in C_c^\infty(\Omega)$ such that $\norm{u-\phi_n}_{a,b,V^+}\to 0$. But
\begin{align*}
\norm{u-\phi_n}_{a,b,V}^2=\norm{u-\phi_n}_{a,b,V^+}^2-\int_\Omega V^-\abs{u-\phi_n}^2\,dx\leq\norm{u-\phi_n}_{a,b,V^+}^2.
\end{align*}
Hence $C_c^\infty(\Omega)$ is also dense in $D(\E^{(N)}_{a,b,V})$, that is, $\E^{(D)}_{a,b,V}=\E^{(N)}_{a,b,V}$.
\end{proof}

\begin{corollary}
If $\E^{(N)}_{a,0,0}=\E^{(D)}_{a,0,0}$, then $\E^{(N)}_{a,b,V}=\E^{(D)}_{a,b,V}$.
\end{corollary}
\begin{proof}
Just combine Theorem \ref{uniqueness_magnetic_forms} and Lemma \ref{negative_part_negligible}.
\end{proof}

We conclude the section giving some examples in which the minimal and maximals Schrödinger form without magnetic field coincide and in which our theorem is thus applicable.

\begin{example}
If $\Omega=\IR^n$, then $-\Delta$ is essentially self-adjoint on $C_c^\infty(\IR^n)$ and therefore $\E^{(N)}_{1,0,0}=\E^{(D)}_{1,0,0}$. This is a classical result in the theory of elliptic differential operators (see \cite{Weid}, Satz 11.15, for example).
\end{example}
This example is a nice illustration of the strength of our method: From this comparably simple result concerning an elliptic differential operator with constant coefficients we can infer $\E^{(N)}_{1,b,V}=\E^{(D)}_{1,b,V}$ under quite weak regularity assumptions on $b$ and $V$. 

For divergence type forms without potential a characterization of $\E^{(N)}=\E^{(D)}$ is given in \cite{RS11} (also compare \cite{GM13} for analog results on weighted manifolds) in terms of the capacity of the boundary: For a measurable set $A\subset \overline\Omega$ the capacity is defined as
\begin{align*}
\mathrm{cap}(A)=\inf\{\norm{u}_{a,0,0}^2\mid u\in D(\E^{(N)}_{a,0,0}), \text{ there is }U\supset A\text{ open: }u|_{U\cap \Omega}=1\}.
\end{align*}
\begin{example}
Let $a\in W^{1,\infty}(\Omega;M_n(\IR)^+)$. Then $\E^{(N)}_{a,0,0}=\E^{(D)}_{a,0,0}$ iff $\mathrm{cap}(\partial \Omega)=0$.
\end{example}

For a bounded domain $\Omega$ the (non-magnetic) forms with Dirichlet and Neumann boundary conditions coincide if the potential is sufficiently singular at the boundary $\partial \Omega$. An estimate in the one-dimensional case follows from Thm. X.10 in \cite{RS2}.

\begin{example}
Assume that $V\in C((0,1))$ is positive and 
\begin{align*}
V(x)\geq \frac  3 4\frac 1{d(x)^2}
\end{align*}
near $0$ and $1$, where $d(x)=\min\{x,1-x\}$. Then $-\Delta+V$ is essentially self-adjoint on $C_c^\infty((0,1))$ and consequently $\E^{(D)}_{1,0,V}=\E^{(N)}_{1,0,V}$.
\end{example}

This result has been generalized to more than one dimension and improved in various directions. We just mention as one example the conditions given in Thm. 2 of \cite{NN09}.

\begin{example}
Let $\Omega\subset\IR^n$ be a bounded domain with $C^2$-boundary. Assume that $V=V_1+V_2$, $V_1\in L^\infty(\Omega)$ and
\begin{align*}
V_2(x)\geq \frac 1{d(x,\partial\Omega)^2}\left(\frac  3 4-\frac 1{\ln d(x,\partial \Omega)^{-1}}-\frac 1 {\ln d(x,\partial\Omega)^{-1}\ln\ln d(x,\partial\Omega)^{-1}}-\dots\right)
\end{align*}
for all $x\in\Omega$ near $\partial \Omega$. Then $-\Delta+V$ is essentially self-adjoint on $C_c^\infty(\Omega)$ and consequently $\E^{(D)}_{1,0,V}=\E^{(N)}_{1,0,V}$.
\end{example}

\appendix
\section{Forms and semigroups}\label{appendix}
In this appendix we present some basic definitions and facts about forms and semigroups on Hilbert spaces. Further information can be found for example in \cite{Kat}, Chapters VI. and IX.

\begin{definition}[Form]
Let $\H$ be a Hilbert space. A densely defined symmetric, sesquilinear form $q$ on $\H$ is map $q\colon D(q)\times D(q)\lra\IC$, where $D(q)\subset \H$ is a dense subspace, such that
\begin{itemize}
\item{}$q(u,v)=\overline{q(v,u)}$,
\item{}$q(\alpha_1 u_1+\alpha_2 u_2,v)=\alpha_1 q(u_1,v)+\alpha_2 q(u_2,v)$
\end{itemize}
for all $\alpha_1,\alpha_2\in\IC,\,u_1,u_2,v\in \H$.\\
A densely defined, symmetric, sesquilinear form $q$ is lower semibounded with lower bound $-\lambda\in\IR$ if
\begin{align*}
q(u,u)\geq-\lambda\langle u,u\rangle_\H
\end{align*}
holds for all $u\in D(q)$.
\end{definition}

Throughout this article, all forms are assumed to be densely defined,  sesquilinear, symmetric and lower semibounded and we will simply call them forms. If we want to emphasize the domain $D(q)$ of $q$, we will write $(q,D(q))$ for the form $q\colon D(q)\times D(q)\lra\IC$.\\
A form $q$ gives rise to a quadratic form on $\H$ via 
\begin{align*}
q\colon \H\lra (-\infty,\infty],\,q(u):=\begin{cases}q(u,u)&\colon u\in D(q),\\\infty &\colon u\notin D(q).\end{cases}
\end{align*}
We shall not distinguish notationally between these two maps.

\begin{lemma}
Let $(q,D(q))$ be a form on $\H$ with lower bound $-\lambda$. Then for every $\alpha>\lambda$,
\begin{align*}
q_\alpha\colon D(q)\times D(q)\lra\IC,\,\langle u,v\rangle_{q,\alpha}=\alpha\langle u,v\rangle_{\H}+q(u,v)
\end{align*}
defines an inner product on $D(q)$. Furthermore, for any $\alpha,\beta>\lambda$, the norms induced by $q_\alpha$ and $q_\beta$ are equivalent.
\end{lemma}
\begin{proof}
Let $\alpha>\lambda$. Since the $\langle\cdot,\cdot\rangle_\H$ and $q$ are sesquilinear and symmetric, so is $q_\alpha$. Moreover, we have for all $u\in D(q)$:
\begin{align*}
q(u,u)=\alpha\langle u,u\rangle_\H+q(u,u)= (\langle u,u\rangle_\H+\lambda q(u,u))+(\alpha-\lambda)\langle u,u\rangle_\H\geq 0
\end{align*}
with equality if and only if $\langle u,u\rangle_\H=0$, that is, $u=0$.\\
Now let $\alpha\geq \beta>\lambda$. Then we have for all $u\in\H$:
\begin{align*}
\beta\langle u,u\rangle_\H+q(u,u)&\leq\alpha\langle u,u\rangle_\H+q(u,u)\\
&=(\alpha-\lambda)\langle u,u\rangle_\H+q(u,u)+\lambda\langle u,u\rangle_\H\\
&\leq (\alpha-\lambda)\langle u,u\rangle_\H+\frac{\alpha-\lambda}{\beta-\lambda}(q(u,u)+\lambda\langle u,u\rangle_\H)\\
&=\frac{\alpha-\lambda}{\beta-\lambda}(q(u,u)+\beta\langle u,u\rangle_\H),
\end{align*}
that is, $q_\alpha$ and $q_\beta$ induce equivalent norms.
\end{proof}

Usually we omit the index $\alpha$ and mean by $\langle \cdot,\cdot\rangle_q$ the inner product $q_\alpha$ for some fixed $\alpha>\lambda$. The induced norm will be denoted by $\norm\cdot_q$.

\begin{definition}[Closed and closable forms]
A form $(q,D(q))$ on $\H$ is called closed if $(D(q),\langle\cdot,\cdot\rangle_q)$ is a Hilbert space.\\
A form $(q^\prime,D(q^\prime))$ is called extension of $(q,D(q))$, $(q,D(q))\subset (q^\prime,D(q^\prime))$ or in short $q\subset q^\prime$, if $D(q)\subset D(q^\prime)$ and $q(u,v)=q^\prime(u,v)$ for all $u,v\in D(q)$.\\
The form $q$ is called closable if it has a closed extension on $\H$.
\end{definition}

\begin{example}
Let $\H$ be a Hilbert space and $T\colon D(T)\lra \H$ a symmetric operator such that there is a $\lambda>0$ such that $\langle Tu,u\rangle\geq -\lambda\norm{u}^2$ for all $u\in\H$. Then the form
\begin{align*}
q\colon D(T)\times D(T)\lra\IC,\,q(u,v)=\langle T u,v\rangle
\end{align*}
is closable. Its closure is denoted by $q_T$.
\end{example}
\begin{proof}
Let $(u_n)$ be a $\|\cdot\|_q$-Cauchy sequence in $D(T)$ that converges to $0$ in $\H$. Then we have
\begin{align*}
\norm{u_n}_q^2&=\abs{\langle u_n,u_n-u_m\rangle_q+\langle u_n,u_m\rangle_q}\\
&\leq \norm{u_n}_q\norm{u_n-u_m}_q+\norm{(T+1+\lambda)u_n}\norm{u_m},
\end{align*}
hence $\norm{u_n}_q\to 0,\,n\to\infty$. Thus, we can embed the completion of $D(T)$ under $\norm\cdot_q$ into $\H$ and the extension of $q$ to the completion is a closed extension of $q$.
\end{proof}

Indeed, all closed forms are of the type $q_T$ for some lower semi-bounded self-adjoint operator $T$, as the next proposition shows.

\begin{proposition}
Let $q$ be a closed form on $\H$. Then 
\begin{align*}
D(T)=\{u\in D(q)\mid \exists w\in \H\;\forall v\in D(q)\colon q(u,v)=\langle w,v\rangle\},\,Tu=w
\end{align*}
defines a lower semi-bounded, self-adjoint operator on $\H$ and $q=q_T$. The operator $T$ is called the generator of $q_T$.
\end{proposition}

\begin{remark}
Notice that we define the generator so that it is lower semi-bounded and will also adapt that convention for semi-groups and resolvents. In some texts the opposite sign convention is chosen so that $T$ is upper semi-bounded, and sometimes even sign conventions are mixed so that the generator of a form is minus the generator of the associated semi-group. This also effects the sign in Kato's inequality---according to our convention, $-\Delta$ is the generator of a form.
\end{remark}

In applications the Hilbert space in questions is often an $L^2$-space. This does not only carry the Hilbert space structure but also an order structure (given by the almost everywhere order). Forms behaving well in a certain sense with respect to the order structure are characterized by the Beurling-Deny criteria. Those forms are intensively studied, see for example \cite{Fu} for an introduction.\\

\begin{definition}[Dirichlet form]
Let $(X,m)$ be a measure space. A closed form $q$ on $L^2(X,m)$ with lower bound $0$ is called positive if it satisfies the first Beurling-Deny criterion:\\
For all $u\in D(q)$, $\abs{u}\in D(q)$ and $q(\abs{u})\leq q(u)$.\\
It is called Dirichlet form if it satiesfies the second Beurling-Deny criterion:\\
For all $u\in D(q)$, $(0\vee v)\wedge 1\in D(q)$ and $q((0\vee v)\wedge 1)\leq q(u)$.
\end{definition}

We will now give a number of examples of Dirichlet forms.

\begin{example}
Let $(X,m)$ be a measure space and $V\colon X\lra [0,\infty)$ measurable. Then
\begin{align*}
D(q)=\{u\in L^2(X,m)\mid V\abs{u}^2\in L^1(X,m)\},\,q(u,v)=\int_X Vu\overline v\,dm
\end{align*}
defines a Dirichlet form on $L^2(X,m)$. Its generator is the multiplication operator $M_V$ on $L^2(X,m)$.
\end{example}

\begin{example}\label{energy-form_domain}
Let $\Omega\subset\IR^n$ be an open subset. Then
\begin{align*}
D(\E)=C_c^\infty(\Omega),\,\E(u,v)=\int_\Omega \langle \nabla u(x),\nabla v(x)\rangle\,dx
\end{align*}
defines a closable form and its closure is a Dirichlet form. The generator of the closure is minus the Laplacian with Dirichlet boundary conditions, see Section \ref{Magnetic Schrödinger operators in Euclidean space} for details (also remember the remark on the sign convention).
\end{example}

\begin{example}
The form defined by
\begin{align*}
D(q)&=\{u\in \ell^2(\IZ^d)\mid \sum_{\abs{m-n}=1}\abs{u(m)-u(n)}^2<\infty\},\\
q(u,v)&=\frac 1 2\sum_{\abs{m-n}=1}(u(m)-u(n))\overline{(v(m)-v(n))}
\end{align*}
is a Dirichlet form on $\ell^2(\IZ^d)$. Its generator is a discrete analog of the Laplacian, see Section \ref{Magnetic Schrödinger operators on graphs} for details on forms of this type.
\end{example}

\begin{definition}[Semigroup]
Let $\H$ be a Hilbert space. A family $(P_t)_{t\geq 0}$ of bounded operators on $\H$ is called a strongly continuous semigroup of symmetric operators of type $\lambda>0$ if the following properties hold:
\begin{itemize}
\item $P_sP_t=P_{s+t},\,P_0=\id$ for all $s,t\geq 0$,
\item $\lim_{t\downto 0}P_t f=f$ for all $f\in\H$,
\item $P_t^\ast=P_t$ for all $t\geq 0$,
\item There is an $M\geq 0$ such that$\norm{P_t}\leq M e^{\lambda t}$ for all $t\geq 0$.
\end{itemize}
The family $(P_t)_{t\geq 0}$ will be called contraction semigroup if $\norm{P_t}\leq 1$ for all $t\geq 0$.\\
In this article, all semigroups will be assumed to be strongly continuous semigroups of symmetric operators of positive type and simply called semigroups.
\end{definition}

\begin{example}
Let $A$ be a self-adjoint operator on $H$ with lower bound $-\lambda$. Then the spectrum $\sigma(A)$ is contained in $[-\lambda,\infty)$ and therefore $e^{-t\cdot}$ is bounded on $\sigma(A)$ by $e^{\lambda t}$ and the operators $(e^{-tA})_{t\geq 0}$, defined via spectral calculus, form a semigroup.
\end{example}

As in the case of quadratic forms, all semigroups are of this form:
\begin{proposition}
Let $(P_t)$ be a semigroup on $\H$. Then
\begin{align*}
D(A)=\{u\in\H\mid\lim_{t\downto 0}\frac 1 t( u- P_t u)\text{ exists}\}, Au=\lim_{t\downto 0}\frac 1 t(u-P_t u)
\end{align*}
defines a lower semi-bounded, self-adjoint operator on $\H$ and $P_t=e^{-tA}$ for all $t\geq 0$.
\end{proposition}

The remark on the sign convention for the generator of a form applies here as well. If the form $q$ has the same generator as the semigroup $(P_t)$, then $q$ is called the form associated with $(P_t)$ and vice versa. The form associated with a semigroup can easily be obtained by the following formula:

\begin{lemma}
Let $(P_t)$ be a semigroup on $\H$. Then the associated form is given by
\begin{align*}
D(q)=\{u\in\H\mid \lim_{t\downto 0}\langle u-P_t u,u\rangle<\infty\},\,q(u,v)=\lim_{t\downto 0}\langle u-P_t u,v\rangle.
\end{align*}
\end{lemma}

As in the case of forms, we are interested in semigroups on $L^2$ that behave well with respect to the order structure.
\begin{definition}[Positivity preserving and Markovian semigroups]
Let $(X,m)$ be a measure space and $(P_t)$ be a semigroup on $L^2(X,m)$.
\begin{itemize}
\item[(a)]The semigroup is called positivity preserving if $P_t L^2_+(X,m)\subset L^2_+(X,m)$.
\item[(b)]The semigroup is called Markovian if $P_t C\subset C$, where $C=\{f\in L^2(X,m)\mid 0\leq f\leq 1\}$.
\end{itemize}
\end{definition}

\begin{example}
Define
\begin{align*}
p_t\colon \IR^n\lra [0,\infty),\,p_t(x)=\frac 1{(4\pi t)^{\frac n 2}}e^{-\frac{\abs{x}^2}{4t}}.
\end{align*}
Then the operators $P_t\colon L^2(\IR^n)\lra L^2(\IR^n), P_t f=p_t\ast f$, $t\geq 0$, form a Markovian semigroup on $L^2(\IR^n)$.\\
Its generator is $-\Delta\colon H^2(\IR^n)\lra L^2(\IR^n)$ and the associated form is the closure of the form $\E$ from Example \ref{energy-form_domain} for $\Omega=\IR^n$.
\end{example}

These properties can be characterized by the associated forms. Indeed, the following more general theorem by Ouhabaz holds (\cite{Ouh96}, Thm. 2.1 and Proposition 2.3, \cite{Ouh99}, Thm. 3; also compare \cite{MVV}, Thm. 2.1).

\begin{proposition}\label{invariance_ouhabaz}
Let $\H$ be a Hilbert space, $C$ a closed, convex subset of $\H$, $P$ the projection onto $C$, $(Q_t)$ a semigroup on $\H$ with generator $T$ and $q$ the associated form with lower bound $-\lambda$. Then the following are equivalent:
\begin{itemize}
\item[(i)]$Q_t C\subset C$ for all $t\geq 0$
\item[(ii)]$\alpha(T+\alpha)^{-1}C\subset C$ for all $\alpha>\lambda$
\item[(iii)]$P(D(q))\subset D(q)$ and $\Re q(Pu,u-P u)\geq 0$ for all $u\in D(q)$
\item[(iv)]$P(D(q))\subset D(q)$ and $\Re q(u,u-Pu)\geq -\lambda \norm{u-Pu}^2$ for all $u\in D(q)$
\end{itemize}
\end{proposition}

A corollary are the well-known Beurling-Deny criteria that connect the notions of positive and Dirichlet forms with those of positivity preserving and Markovian forms.

\begin{corollary}\label{Beurling_Deny}
Let $(X,m)$ be a measure space, $(P_t)$ a semigroup on $L^2(X,m)$ and $q$ the associated quadratic form.
\begin{itemize}
\item[(a)]The semigroup $(P_t)$ is positivity preserving if and only if $q$ is positive.
\item[(b)]The semigroup $(P_t)$ is Markovian if and only if $q$ is a Dirichlet form.
\end{itemize}
\end{corollary}

%
%
%
%
\addcontentsline{toc}{section}{References}
\bibliography{mybib}{}

\begin{thebibliography}{HKMW13}

\bibitem[ATH15]{AT15}
C.~Ann{\'e} and N.~Torki-Hamza.
\newblock The {G}auss-{B}onnet operator of an infinite graph.
\newblock {\em Anal. Math. Phys.}, 5(2):137--159, 2015.

\bibitem[B{\'e}r86]{Ber}
P.~H. B{\'e}rard.
\newblock {\em {Spectral Geometry: Direct and Inverse Problems}}.
\newblock Lecture Notes in Mathematics. Springer-Verlag, 1986.

\bibitem[BMS02]{BMS02}
M.~Braverman, O.~Milatovic, and M.~Shubin.
\newblock Essential self-adjointness of schrödinger-type operators on
  manifolds.
\newblock {\em Russian Mathematical Surveys}, 57(4):641, 2002.

\bibitem[FLW14]{FLW}
R.~L. Frank, D.~Lenz, and D.~Wingert.
\newblock Intrinsic metrics for non-local symmetric {D}irichlet forms and
  applications to spectral theory.
\newblock {\em J. Funct. Anal.}, 266(8):4765--4808, 2014.

\bibitem[FOT94]{Fu}
M.~Fukushima, Y.~Oshima, and M.~Takeda.
\newblock {\em {Dirichlet Forms and Symmetric Markov Processes}}.
\newblock De Gruyter Studies in Mathematics Series. De Gruyter, 1994.

\bibitem[Gaf51]{Gaf51}
M.~P. Gaffney.
\newblock The harmonic operator for exterior differential forms.
\newblock {\em Proc. Nat. Acad. Sci. U. S. A.}, 37:48--50, 1951.

\bibitem[GM13]{GM13}
A.~Grigor'yan and J.~Masamune.
\newblock Parabolicity and stochastic completeness of manifolds in terms of the
  {G}reen formula.
\newblock {\em J. Math. Pures Appl. (9)}, 100(5):607--632, 2013.

\bibitem[HKLW12]{HKLW}
S.~Haeseler, M.~Keller, D.~Lenz, and R.~Wojciechowski.
\newblock {Laplacians on infinite graphs: Dirichlet and Neumann boundary
  conditions}.
\newblock {\em Journal of Spectral Theory}, 2:397--432, 2012.

\bibitem[HKMW13]{HKMW}
X.~Huang, M.~Keller, J.~Masamune, and R.~Wojciechowski.
\newblock {A note on self-adjoint extensions of the Laplacian on weighted
  graphs}.
\newblock {\em Journal of Functional Analysis}, 265(8):1556--1578, 2013.

\bibitem[HL15]{HL15}
B.~Hua and Y.~Lin.
\newblock Stochastic completeness for graphs with curvature dimension
  conditions.
\newblock {\em arXiv preprint arXiv:1504.00080}, 2015.

\bibitem[HR15]{HR15}
M.~{Hinz} and L.~{Rogers}.
\newblock {Magnetic fields on resistance spaces}.
\newblock {\em ArXiv e-prints}, January 2015.

\bibitem[HS04]{HS04}
D.~Hundertmark and B.~Simon.
\newblock A diamagnetic inequality for semigroup differences.
\newblock {\em JOURNAL FUR DIE REINE UND ANGEWANDTE MATHEMATIK.}, pages
  107--130, 2004.

\bibitem[HSU77]{HSU}
H.~Hess, R.~Schrader, and D.~A. Uhlenbrock.
\newblock Domination of semigroups and generalization of {K}ato's inequality.
\newblock {\em Duke Math. J.}, 44(4):893--904, 1977.

\bibitem[HT13]{HT13}
M.~Hinz and A.~Teplyaev.
\newblock Dirac and magnetic {S}chr\"odinger operators on fractals.
\newblock {\em J. Funct. Anal.}, 265(11):2830--2854, 2013.

\bibitem[IN90]{IN}
G.~Isac and A.~B. N{\'e}meth.
\newblock Every generating isotone projection cone is latticial and correct.
\newblock {\em J. Math. Anal. Appl.}, 147(1):53--62, 1990.

\bibitem[Kat72]{Kat72}
T.~Kato.
\newblock Schr{\"o}dinger operators with singular potentials.
\newblock {\em Israel Journal of Mathematics}, 13(1-2):135--148, 1972.

\bibitem[Kat95]{Kat}
T.~Kato.
\newblock {\em Perturbation Theory for Linear Operators}.
\newblock Classics in Mathematics. Springer Berlin Heidelberg, 1995.

\bibitem[KL10]{KL10}
M.~Keller and D.~Lenz.
\newblock Unbounded {Laplacians} on graphs: basic spectral properties and the
  heat equation.
\newblock {\em Mathematical Modelling of Natural Phenomena}, 5(04):198--224,
  2010.

\bibitem[KL12]{KL12}
M.~Keller and D.~Lenz.
\newblock {Dirichlet forms and stochastic completeness of graphs and
  subgraphs}.
\newblock {\em Journal für die reine und angewandte Mathematik}, 666:189--223,
  2012.

\bibitem[LL97]{LL97}
E.~H. Lieb and M.~Loss.
\newblock {\em Analysis}, volume~14 of {\em Graduate Studies in Mathematics}.
\newblock American Mathematical Society, Providence, RI, 1997.

\bibitem[LZ71]{LZ}
W.~A.~J. Luxemburg and A.~C. Zaanen.
\newblock {\em Riesz spaces. {V}ol. {I}}.
\newblock North-Holland Publishing Co., Amsterdam-London; American Elsevier
  Publishing Co., New York, 1971.
\newblock North-Holland Mathematical Library.

\bibitem[Man01]{Man01}
A.~Manavi.
\newblock {\em {Zur Störung von dominierten $C_0$-Halbgruppen auf
  Banachfunktionenräumen mit ordnungsstetiger Norm und sektoriellen Formen mit
  singulären komplexen Potentialen}}.
\newblock PhD thesis, TU Dresden, 2001.

\bibitem[Mor62]{M62}
J.-J. Moreau.
\newblock D\'ecomposition orthogonale d'un espace hilbertien selon deux c\^ones
  mutuellement polaires.
\newblock {\em C. R. Acad. Sci. Paris}, 255:238--240, 1962.

\bibitem[MS74]{MS74}
J.~W. Milnor and J.~D. Stasheff.
\newblock {\em Characteristic classes}.
\newblock Princeton University Press, Princeton, N. J.; University of Tokyo
  Press, Tokyo, 1974.
\newblock Annals of Mathematics Studies, No. 76.

\bibitem[MT14]{MT14}
O.~Milatovic and F.~Truc.
\newblock Self-adjoint extensions of discrete magnetic {S}chr\"odinger
  operators.
\newblock {\em Ann. Henri Poincar\'e}, 15(5):917--936, 2014.

\bibitem[MT15]{MT15}
O.~Milatovic and F.~Truc.
\newblock Maximal accretive extensions of {S}chr\"odinger operators on vector
  bundles over infinite graphs.
\newblock {\em Integral Equations Operator Theory}, 81(1):35--52, 2015.

\bibitem[MVV05]{MVV}
A.~Manavi, H.~Vogt, and J.~Voigt.
\newblock Domination of semigroups associated with sectorial forms.
\newblock {\em Journal of Operator Theory}, 54(1):9--26, 2005.

\bibitem[N{\'e}m03]{Nem03}
A.~B. N{\'e}meth.
\newblock Characterization of a {H}ilbert vector lattice by the metric
  projection onto its positive cone.
\newblock {\em J. Approx. Theory}, 123(2):295--299, 2003.

\bibitem[NN09]{NN09}
G.~Nenciu and I.~Nenciu.
\newblock On confining potentials and essential self-adjointness for
  {S}chr\"odinger operators on bounded domains in {$\mathbb R^n$}.
\newblock {\em Ann. Henri Poincar\'e}, 10(2):377--394, 2009.

\bibitem[Ouh96]{Ouh96}
E.~Ouhabaz.
\newblock Invariance of closed convex sets and domination criteria for
  semigroups.
\newblock {\em Potential Analysis}, 5(6):611--625, 1996.

\bibitem[Ouh99]{Ouh99}
E.~Ouhabaz.
\newblock {$L^p$} contraction semigroups for vector valued functions.
\newblock {\em Positivity}, 3(1):83--93, 1999.

\bibitem[Pen76]{Pen76}
R.~C. Penney.
\newblock Self-dual cones in {H}ilbert space.
\newblock {\em J. Functional Analysis}, 21(3):305--315, 1976.

\bibitem[RS75]{RS2}
M.~Reed and B.~Simon.
\newblock {\em Methods of modern mathematical physics. 2. Fourier analysis,
  self-adjointness}, volume~2.
\newblock Elsevier, 1975.

\bibitem[RS11]{RS11}
D.~W. Robinson and A.~Sikora.
\newblock Markov uniqueness of degenerate elliptic operators.
\newblock {\em Ann. Sc. Norm. Super. Pisa Cl. Sci. (5)}, 10(3):683--710, 2011.

\bibitem[Shu01]{Sch01}
M.~Shubin.
\newblock Essential self-adjointness for semi-bounded magnetic {S}chr\"odinger
  operators on non-compact manifolds.
\newblock {\em J. Funct. Anal.}, 186(1):92--116, 2001.

\bibitem[Sim77]{Sim77}
B.~Simon.
\newblock {An Abstract Kato’s inequality for generators of positivity
  preserving semigroups}.
\newblock {\em Indiana University Mathematics Journal}, 26(6), 1977.

\bibitem[Sim79a]{Sim79b}
B.~Simon.
\newblock Kato's inequality and the comparison of semigroups.
\newblock {\em Journal of Functional Analysis}, 32(1):97--101, 1979.

\bibitem[Sim79b]{Sim79}
B.~Simon.
\newblock {Maximal and minimal Schrödinger forms}.
\newblock {\em J. Operator Theory}, 1:37--47, 1979.

\bibitem[Wei00]{Weid}
J.~Weidmann.
\newblock {\em {Lineare Operatoren in Hilberträumen. Teil I Grundlagen}}.
\newblock B. G. Teubner, 2 edition, 2000.

\bibitem[Wer07]{Wer}
D.~Werner.
\newblock {\em Funktionalanalysis}.
\newblock Springer-Verlag, Berlin, extended edition, 2007.

\end{thebibliography}
\bibliographystyle{alpha}

\textsc{Institut für Mathematik, Friedrich-Schiller-Universität Jena, 07737 Jena, Germany}\\
\textit{E-mail address:} \href{mailto:melchior.wirth@uni-jena.de}{melchior.wirth@uni-jena.de}\\
\textit{URL:} \url{http://www.analysis-lenz.uni-jena.de/Team/Melchior+Wirth.html}

\end{document}